\documentclass{amsart}
\usepackage{amsmath,amssymb,graphicx}
\usepackage[colorlinks=true]{hyperref}

\numberwithin{equation}{section}

\newtheorem{theorem}{Theorem}[section]
\newtheorem{lemma}[theorem]{Lemma}
\newtheorem{definition}[theorem]{Definition}
\newtheorem{proposition}[theorem]{Proposition}
\newtheorem{corollary}[theorem]{Corollary}

\everymath{\displaystyle}

\begin{document}

\title[Prescribing curvatures on surfaces with singularities]{Prescribing curvatures on surfaces with conical singularities and corners}

\author{Luca Battaglia}
\address{Luca Battaglia\\
Universit\`a degli Studi Roma Tre\\
Dipartimento di Matematica e Fisica\\
Largo S. Leonardo Murialdo 1\\
00146 Roma, Italy}
\email{luca.battaglia@uniroma3.it}

\author{Francisco J. Reyes-S\'anchez}
\address{Francisco J. Reyes-S\'anchez\\
Universidad de Granada\\
Departamento de An\'alisis Matem\'atico\\
Campus Fuentenueva\\
18071 Granada, Spain}
\email{fjreyes@ugr.es}

\thanks{L.B. has been partially supported by MUR-PRIN-2022AKNSE4 and INDAM-GNAMPA project ``Problemi di doppia curvatura su variet\`a a bordo e legami con le EDP ti tipo ellittico''. F.J.R.S. has been supported by a PhD fellowship (PRE2021-099898) linked to the \emph{IMAG-Maria de Maeztu} Excellence Grant CEX2020-001105-M funded by MICIN/AEI.}



\begin{abstract}
This paper is concerned with the problem of prescribing Gaussian curvature $K$ and geodesic curvature $h$ in a compact surface with boundary $\Sigma$ with conical singularities $\{p_1,\dots,p_n\}$ and corners $\{q_1,\dots,q_n\}$. This is equivalent to solving the Liouville-type equation:
$$\left\{\begin{array}{ll}-\Delta u+2K_0=2Ke^u-4\pi\sum_{i=1}^n\alpha_i\left(\delta_{p_i}-\frac1{|\Sigma|}\right)-2\pi\sum_{j=1}^m\beta_j\left(\delta_{q_j}-\frac1{|\Sigma|}\right)&\mbox{in }\Sigma\\\partial_\nu u+2h_0=2he^{u/2}&\mbox{on }\partial\Sigma\end{array}\right.,$$
where $K_0,h_0$ are the pre-existing Gaussian curvature and the geodesic curvature, respectively, and $\alpha_i,\beta_j>-1$ are given.\\
Solutions are obtained using a new variational formulation, first introduced in \cite{bls2} for the regular counterpart of the problem and extended here to the singular case.\\
As far as we know, this is the first result for the problem of prescribed curvatures in surfaces with the two types of singularities. Key ingredients are a blow-up analysis around a sequence of points different from local maxima and Morse index estimates.

\end{abstract}

\maketitle

\section{Introduction}

Let $(\Sigma,g_0)$ be a closed Riemannian surface. A classical problem in geometric analysis is given by the \emph{prescribed Gaussian curvature problem}, that is: given a smooth function $K:\Sigma\to\mathbb R$, find another metric $g=e^ug_0$ being conformally equivalent to $g_0$ with Gaussian curvature equal to $K(x)$ at each point $x\in\Sigma$. Such a problem is equivalent to solving the following Liouville-type equation
\begin{equation}\label{gauss}
-\Delta u+2K_0=2Ke^u\qquad\mbox{in }\Sigma,
\end{equation}
where $\Delta$ is the Laplace-Beltrami operator associated to the metric $g_0$ and $K_0$ is the Gaussian curvature with respect to the metric $g_0$. The problem \eqref{gauss} has been widely studied for decades; we refer to Chapter 6 in \cite{a} (and its references) for a survey.\\
A natural extension of problem \eqref{gauss} to compact surfaces with boundary is the \emph{double curvature prescription problem}: given smooth functions $K:\Sigma\to\mathbb R$ and $h:\partial\Sigma\to\mathbb R$, we want to find a conformal metric $g=e^ug_0$ such that the Gaussian curvature with respect to $g$ equals $K$ pointwise on $\Sigma$ and its geodesic curvature coincides with $h$ pointwise on $\partial\Sigma$. The equivalent formulation in terms of partial differential equations is given by the following nonlinear Neumann problem on $\Sigma$:
\begin{equation}\label{eqreg}
\left\{\begin{array}{ll}-\Delta u+2K_0=2Ke^u&\mbox{in }\Sigma\\\partial_\nu u+2h_0=2he^{u/2}&\mbox{on }\partial\Sigma\end{array}\right.,
\end{equation}
with $\partial_\nu$ being the outer normal derivative and $h_0$ being the geodesic curvature with respect to $g_0$.\\
Despite its relevance, problem \eqref{eqreg} has been significantly less explored compared to \eqref{gauss}. Some results are available for particular cases: specifically, the case $h = 0$ has been treated in \cite{cy}, while the case $K = 0$ has been studied in \cite{cl,ll,lh}. Blow-up behavior of solutions has also been studied, see \cite{bwz,dmr}. When $K$ and $h$ are constants, explicit solutions have been classified when the domain is a disc or an annulus; see \cite{hw,j}. In addition, classification results in the half-plane setting can be found in \cite{gm,lz,z}.\\
However, not much is known about the general case where both curvatures are non-constant functions. A variational approach has been explored in \cite{cr,lsmr,bls2}, while some asymptotic properties have been examined in \cite{jlsmr,bmp1,bmp2}. Further progress in this direction remains an open challenge.\\

This paper is concerned with generalizing the problem \eqref{eqreg} to a metric $g$ with singularities such as finitely many given \emph{conical points} in $\Sigma$ or \emph{corners} in $\partial\Sigma$. We say that $p$ is a conical singularity of order $\alpha>-1$, or angle $2\pi(1+\alpha)$, if there is a local coordinate system $\Xi:B_\delta(p)\to\mathbb R^2$ with $\Xi(p)=0$ and $g$ has the form
\begin{equation}\label{conical1}
g(x)=|\Xi(x)|^{2\alpha}e^{v(x)}g_e(x),
\end{equation}
where $v$ is a smooth function and $g_e$ denotes the standard Euclidean metric. This implies that $(\Sigma,g)$ is locally isometric around $p$ to an Euclidean cone with angle $2\pi(1+\alpha)$.\\
Similarly, $q\in\partial\Sigma$ is a \emph{corner} of order $\beta>-1$, or angle $\pi(1+\beta)$, if in a suitable local coordinate system $\Xi$ around $q$ the metric $g$ has the form
\begin{equation}\label{conical2}
g(x)=|\Xi(x)|^{2\beta}e^{v(x)}g_e(x);
\end{equation}
in this case, $(\Sigma,g)$ is locally isometric to the planar sector
$$\{(\rho\cos\theta,\rho\sin\theta):\,\rho\ge0,\,0\le\theta\le\pi(1+\beta)\}.$$
Fixing some conical singularities $p_1,\dots,p_n\in\Sigma$ of orders $\alpha_1,\dots,\alpha_n$, respectively, and some corners $q_1,\dots,q_m\in\partial\Sigma$ of orders $\beta_1,\dots,\beta_m$, respectively, the equation that must be satisfied by the logarithm of the conformal factor is now
$$\left\{\begin{array}{ll}-\Delta u+2K_0=2Ke^u-4\pi\sum_{i=1}^n\alpha_i\left(\delta_{p_i}-\frac1{|\Sigma|}\right)-2\pi\sum_{j=1}^m\beta_j\left(\delta_{q_j}-\frac1{|\Sigma|}\right)&\mbox{in }\Sigma\\\partial_\nu u+2h_0=2he^{u/2}&\mbox{on }\partial\Sigma\end{array}\right.,$$
where $\delta_p$ denotes the Dirac delta function with pole at $p$, and $|\Sigma|$ is the area of $\Sigma$.\\
The problem may be de-singularized by the substitution
$$u:=v+4\pi\left(\sum_{j=1}^n\alpha_iG_{p_i}+\frac12\sum_{i=1}^m\beta_jG_{q_j}\right),$$
where $G_p$ is the Green's function of $-\Delta$ with pole at $p$, i.e. the solution to
$$\left\{\begin{array}{ll}
-\Delta G_p=\delta_p-\frac1{|\Sigma|}&\mbox{in }\Sigma\\
\partial_\nu G_p=0&\mbox{on }\partial\Sigma\\
\int_\Sigma G_p=0
\end{array}\right.,$$
This leads to the following problem:
$$\left\{\begin{array}{ll}
-\Delta u+2K_0=2\widetilde Ke^u&\mbox{in }\Sigma\\\partial_\nu u+2h_0=2\widetilde he^{u/2}&\mbox{on }\partial\Sigma\end{array}\right.,$$
where
\begin{equation}\label{kh}
\widetilde K=Ke^{-4\pi\left(\sum_{i=1}^n\alpha_iG_{p_i}+\frac12\sum_{j=1}^m\beta_jG_{q_j}\right)},\,\,\,\,\widetilde h=he^{-2\pi\left(\sum_{i=1}^n\alpha_iG_{p_i}+\frac12\sum_{j=1}^m\beta_jG_{q_j}\right)};
\end{equation}
we point out that, in view of the logarithmic singularity of $G_p$, the local behavior of $\widetilde K,\widetilde h$ at each singular point is given by
\begin{eqnarray*}
\widetilde K=d(\cdot,p_i)^{2\alpha_i}\widehat K&\mbox{around}&p_i\\
\widetilde K=d(\cdot,q_j)^{2\beta_j}\widehat K&\mbox{around}&q_j\\
\widetilde h=d(\cdot,q_j)^{\beta_j}\widehat h&\mbox{around}&q_j,
\end{eqnarray*}
with $\widehat K,\widehat h$ smooth and bounded from below by positive constants, consistently with \eqref{conical1}, \eqref{conical2}.\\

By integrating both sides of the equation and applying the Gauss-Bonnet theorem, we obtain
\begin{equation}\label{gaussbonnet}
\int_\Sigma\widetilde Ke^u+\int_{\partial\Sigma}\widetilde he^{u/2}=2\pi\chi,
\end{equation}
where $\chi$ represents the singular Euler characteristic of $\left(\Sigma,\underline\alpha,\underline\beta\right)$, first introduced by Troyanov in \cite{tr}, defined by
\begin{equation}\label{chi}
\chi:=\chi\left(\Sigma,\underline\alpha,\underline\beta\right)=\chi(\Sigma)+\sum_{i=1}^n\alpha_i+\frac12\sum_{j=1}^m\beta_j,
\end{equation}
and $\chi(\Sigma)$ is the classical Euler characteristic of $\Sigma$.\\
As a preliminary, let us remark that, without singularities, we can always prescribe zero geodesic curvature and constant Gaussian curvature (see for instance Proposition 3.1 in \cite{lsmr}). Due to this, we may assume that $h_0=0$ and $K_0=\frac{2\pi\chi}{|\Sigma|}$, where $|\Sigma|$ is the measure of the surface area of $\Sigma$, so that the problem simplifies to:
\begin{equation}\label{eqsing}
\left\{\begin{array}{ll}
-\Delta u+\frac{4\pi\chi}{|\Sigma|}=2\widetilde Ke^u&\mbox{in }\Sigma\\
\partial_\nu u=2\widetilde he^{u/2}&\mbox{on }\partial\Sigma
\end{array}\right..
\end{equation}
As we will see in Section \ref{meanfield}, by re-writing this problem as a \emph{mean-field equation}, solutions to \eqref{eqsing} can be interpreted as critical points on the space
$$\overline H^1(\Sigma):=\left\{u\in H^1(\Sigma):\int_\Sigma u=0\right\}$$
of the energy functional:
\begin{equation}\label{j}
\mathcal J(u)=\frac12\int_\Sigma|\nabla u|^2-\mathcal F_\chi\left(\int_\Sigma\widetilde Ke^u,\int_{\partial\Sigma}\widetilde he^{u/2}\right),
\end{equation}
for some smooth $\mathcal F_\chi(A,B)$.\\
This new approach is inspired by the work of \cite{bls2} on prescribing Gaussian and geodesic curvature on smooth surfaces with boundary, without singularities. It extends the well-known mean field model for the prescribed Gaussian curvature problem on closed surfaces, a model that has been widely applied in this context, as seen in the seminal work \cite{dm}.\\
The introduction of this formulation, both for problem \eqref{eqsing} and its counterpart on closed surfaces, is motivated by to its better adaptability to variational methods. Furthermore, it is invariant under the addition of constants. In fact, the mean field formulation allows us to find solutions as either absolute minima or min-max critical points of the energy functional through the use of Trudinger-Moser inequalities. This aspect will be explored in greater depth in Sections \ref{chipos} and \ref{minmaxsol}.\\

In the study of the existence of solutions to the problem \eqref{eqsing}, a key role is played by the \emph{Trudinger constant} of the singular surface, denoted by $\tau$. This constant is defined as
\begin{equation}\label{tau}
\tau:=\tau\left(\underline\alpha,\underline\beta\right)=\min\left\{1, 2+2\min_{i=1,\dots,n}\alpha_i,1+\min_{j=1,\dots,m}\beta_j\right\},
\end{equation}
and together with the \emph{singular Euler characteristic} $\chi$, heavily influences the study of this problem. This is basically due to Trudinger-Moser inequalities (see Proposition \ref{tm}) which allow us to establish whether the energy functional $\mathcal J$ is bounded from below and coercive, depending on which is larger between $\tau$ and $\chi$, thus yielding minimizing solutions. Following \cite{tr}, we have the following definition:
\begin{definition}
The singular surface $\left(\Sigma,\underline\alpha,\underline\beta\right)$ is classified as:
\begin{itemize}
\item\emph{subcritical} if $0<\chi<\tau$;
\item\emph{critical} if $\chi=\tau$;
\item\emph{supercritical} if $\chi>\tau$.
\end{itemize}
\end{definition}
In the \emph{subcritical} case, Trudinger-Moser inequalities can be employed to show the existence of solutions, as the energy functional \eqref{j} is coercive. In the \emph{critical} and \emph{supercritical} cases, coercivity does not hold anymore in $\overline H^1(\Sigma)$, but we can recovery coercivity under additional assumptions of symmetry on the curvatures $K$ and $h$. Specifically, we assume that $K$ and $h$ are invariant under the action of an isometry group $\mathcal G$ of $\overline\Sigma$ with no boundary fixed points, namely:
\begin{equation}\tag{$\mathcal G$}\label{sym}
\begin{array}{c}
\mathrm{Fix}(\mathcal G)\cap\partial\Sigma=\emptyset;\\
\nonumber K(x)=K(gx)\qquad\mbox{for a.e. }x\in\Sigma,\,\forall g\in\mathcal G;\\
\nonumber h(x)=h(gx)\qquad\mbox{for a.e. }x\in\Sigma,\,\forall g\in\mathcal G.
\end{array}
\end{equation}
Symmetry allows us to search for solutions in the subspace of $\overline H^1(\Sigma)$ consisting on $\mathcal G$-symmetric functions:
$$\widetilde H_{\mathcal G}:=\left\{u\in\overline H^1(\Sigma):u(gx)=u(x)\mbox{ for a.e. }x\in\Sigma,\,\forall g\in\mathcal G\right\}.$$
On such a space, one can obtain improved Trudinger-Moser inequalities, which allow us to recover coercivity for $\mathcal J$ (see Proposition \ref{tmimpr}). The role of symmetry was first highlighted in \cite{mo} and has been further explored in recent works on curvature prescription problems, such as \cite{cr,lsrsr}.\\
However, due to the presence of singularities, we must discuss whether $\mathcal G$ has fixed points and, if it does, whether they are conical singularities. We emphasize that this is a new aspect compared to the problem without singularities, originally considered in \cite{bls2}, because in the regular case the only situation where $\chi>0$ is critical, and it corresponds to simply connected surfaces (see \eqref{chi}), where admissible symmetries have one fixed point.\\
Based on this, we got the following result:
\begin{theorem}\label{chi>0}
Assume that $\chi>0$ and $K\ge0,K\not\equiv0$. Then, problem \eqref{eqsing} has a solution if one of the following conditions holds:
\begin{itemize}
\item $\left(\Sigma,\underline\alpha,\underline\beta\right)$ is subcritical;
\item $\left(\Sigma,\underline\alpha,\underline\beta\right)$ is critical, \eqref{sym} holds and
\begin{eqnarray*}
\mbox{either }&&\mathrm{Fix}(\mathcal G)\cap\{p_1,\ldots,p_n\}=\emptyset,\\
\mbox{or }&&\tau<2+2\min_{i:\,p_i\in\mathrm{Fix}(\mathcal G)}\alpha_i;
\end{eqnarray*}
\item $\left(\Sigma,\underline\alpha,\underline\beta\right)$ is supercritical, \eqref{sym} holds and
\begin{eqnarray*}
\mbox{either }&&\mathrm{Fix}(\mathcal G)=\emptyset\quad\mbox{and }\chi<k\tau,\\
\mbox{or }&&\chi<\min\left\{2,2+2\min_{i:\,p_i\in\mathrm{Fix}(\mathcal G)}\alpha_i,k\tau\right\};
\end{eqnarray*}
here,
\begin{equation}\label{k}
k:=\min\{\#\mathcal Gx:\,x\in\Sigma\setminus\mathrm{Fix}(\mathcal G)\}\ge2.
\end{equation}
\end{itemize}
\end{theorem}
We point out that we are making a sign assumption only on the prescribed interior curvature $K$ and no assumptions at all on the prescribed boundary curvature $h$.\\

In the case $\chi=0$, we have two distinct scenarios, as discussed in \cite{bls2}: when $\widetilde K$ is positive, $\mathcal F_\chi$ is negative everywhere, making the energy functional coercive so that the problem has a minimizing solution. Conversely, when $\widetilde K$ is negative, to deal with the nonlinear term and obtain minimizers we must impose conditions on the scale invariant function $\mathfrak D:\partial\Sigma\to\mathbb R$ defined by
\begin{equation}\label{d}
\mathfrak D(x):=\frac{h(x)}{\sqrt{|K(x)|}}.
\end{equation}
Thanks to this the following result can then be stated:
\begin{theorem}\label{chi=0}
Assume $\chi=0$. Then, problem \eqref{eqsing} has a solution if one of the following conditions holds:
\begin{itemize}
\item $K(x)>0$ for some $x\in\Sigma$ and $h(y)\le0$ for all $y\in\partial\Sigma$, with $ \not\equiv0$;
\item $K(x)\ge0$ for all $x\in\Sigma$, $K\not\equiv0$, and $\int_{\partial\Sigma}h<0$;
\item $K(x)\le0$ for all $x\in\Sigma$, $K\not\equiv0$, and $h(y)\ge0,\mathfrak D(y)<1$ for all $y\in\partial\Sigma$, $h\not\equiv0$.
\end{itemize}
\end{theorem}\

In the case $\chi<0$, minimizing solutions can be obtained under a specific assumption involving \eqref{d}, similar to Theorem \ref{chi=0}. This yields the following result, which extends \cite[Theorem 1.3]{bls2} to the singular setting:
\begin{theorem}\label{chi<0}
Assume $\chi<0$, $ K\le0$ for all $x\in\Sigma$, $ K\not\equiv0$, $\int_\Sigma\log|K|>-\infty$, and $\mathfrak D(y)<1$ for all $y\in\partial\Sigma$. Then, problem \eqref{eqsing} has a solution.
\end{theorem}

Quite remarkably, singularities do not play any role in Theorem \ref{chi=0} nor in Theorem \ref{chi<0}. This is essentially due to the fact that the ratio \eqref{d} is invariant when $h$ is replaced with $\widetilde h$ and $K$ is replaced with $\widetilde K$, in light of the definition in \eqref{kh}.\\

The next natural step is to remove the symmetry assumptions in Theorem \ref{chi>0}. This is addressed through a strategy that involves a topological analysis of the sublevel sets of \eqref{j}: 
$$\mathcal J^a=\left\{u\in \overline H^1(\Sigma):\,\mathcal J(u)\le a\right\}.$$
This approach, originally introduced in \cite{dm}, has become a standard technique for handling mean-field equations. A key objective is to show that low sublevels are not contractible. In fact, when $\mathcal{J}(u)\ll0$, the measure 
$$\frac{\widetilde Ke^u}{\int_\Sigma\widetilde Ke^u}$$
tends to concentrate around a finite number of points, a phenomenon that can be rigorously established using localized versions of the Trudinger-Moser inequality.\\
A fundamental consequence of this concentration behavior is that the topology of low sublevels is inherited from the space of finitely supported measures, known as barycenters. This provides the framework for proving the existence of min-max type solutions. While this technique has been widely applied in closed surfaces, to the best of our knowledge, this is the first time it has been used to prescribe two non-vanishing curvatures on surfaces with boundary.\\
Given that the Palais-Smale condition is not known for mean-field type problems, we employ a variation of the \emph{monotonicity trick}, first introduced by Struwe in \cite{s}. In simple terms, we introduce an auxiliary parameter for which the Palais-Smale condition holds for almost every value, then we consider a sequence of such parameters satisfying the Palais-Smale and attempt to pass to the limit, for which we need the compactness of solutions to the corresponding auxiliary problems. Since this approach requires monotonicity with respect to the parameter, its choice is particularly delicate due to the new nonlinear term $\mathcal F_\chi$, which involves both interior and boundary integrals.\\
This method allows us to bypass the Palais-Smale condition, provided that $4\pi \chi$ does not belong to the discrete set 
\begin{eqnarray}
\label{gamma}\Gamma=\Gamma\left(\underline\alpha,\underline\beta\right)&:=&\left\{4\pi k+8\pi\sum_{i\in I}(1+\alpha_i)+4\pi\sum_{j\in J}(1+\beta_j):\right.\\
\nonumber&&\left.\phantom{\sum_{j\in J}}k\in\mathbb N_0,\,I\subset\{1,\dots,n\},\,J\subset\{1,\dots,m\}\right\}.
\end{eqnarray}
Solutions to \eqref{eqsing} may have new and complex blow-up phenomena, compared to classical Liouville-type problems, especially in the case when $K$ and $h$ have opposite signs, or may change sign, since the \emph{masses} $\int_\Sigma\widetilde Ke^u$ and $\int_{\partial\Sigma} \widetilde he^{u/2}$ might both diverge to infinity with different signs; concentration phenomena have not been classified yet in full generality.\\
To address this issue, we restrict ourselves to sequences with finite Morse index; thanks to a result by Fang and Ghoussoub \cite{fg}, we suffice to consider the finite Morse index case because we are using a finite-dimensional min-max argument. In fact, one of the main novelties of this work is the identification of a relationship between the Morse indices of two different variational formulations of the problem. This link is used to exploit the fact that the direct one is particularly suitable for this specific analysis (see Subsection \ref{relation}).\\
Since our approach relies on a min-max argument, compactness properties are essential. A detailed blow-up analysis is provided in Section \ref{blowup}. \\

In order to prove the existence of conformal metrics in the supercritical case, we assume that the singularities $\alpha_i,\beta_j$ satisfy the following inequalities:
\begin{equation}\label{alphabetacond}
\alpha_i\ge-\frac12,\,\forall i=1,\dots,n,\qquad\beta_j\ge0,\,\forall j\in1,\dots, m.
\end{equation}
By \eqref{tau}, one can verify that this assumption is equivalent to assuming the Trudinger constant to be equal to $1$, hence to get a Trudinger-Moser inequality with the same constant.\\
We also need $\Sigma$ to be multiply connected, in order to apply a topological argument similar 	to \cite{bdm,bjmr,bjwy}.\\
We therefore get the following result, which is proved in Subsection \ref{proof}:
\begin{theorem}\label{minmax}
Let $\Sigma$ be a supercritical singular surface with conical singularities and/or corners as in \eqref{alphabetacond}. Assume that:
\begin{itemize}
\item $K>0$ on $\overline\Sigma$;
\item The surface's boundary has at least two components:
$$\partial\Sigma=\bigcup_{k=1}^l\gamma_k\qquad l\ge2;$$
\item There exists at least one boundary component without corners, w.l.o.g.
$$q_j\notin\gamma_1\quad\mbox{for all }j\in\{1,\dots,m\}.$$ 
\end{itemize}
If $4\pi\chi\notin\Gamma$, then there exists a conformal metric with Gaussian curvature $\widetilde K$ and geodesic curvature $\widetilde h$.
\end{theorem}
As for Theorem \ref{chi>0}, we emphasize that this result only needs a mild sign assumption on $K$ and no assumption at all on $h$.

\subsection{Notations}

\

An open ball of center $x$ and radius $r$ will be denoted by $B_r(x)$, while $A_{r,R}(x)$ stands for an open annulus of center $p$ and radii $0<r<R$. We also denote upper half-balls as
$$B^+_r(x):=\left\{(s,t)\in B_r(x)\subset\mathbb R^2:t\ge0\right\},$$
and the upper half-plane as
$$\mathbb R^2_+=\left\{(s,t)\in\mathbb R^2:t\ge0\right\}.$$
Given a set $\Omega\subset\overline\Sigma$ and $\delta>0$, we denote the $\delta$-neighborhood of $\Omega$ as
$$\Omega^\delta=\left\{x\in\overline\Sigma: d(x,\Omega)<r\right\}.$$
The \emph{average} of $u\in H^1(\Sigma)$ will be denoted as
$$\overline u:=\frac1{|\Sigma|}\int_\Sigma u.$$
We denote the positive and negative parts of a real number $t$ respectively as
$$t_+:=\max\{t,0\},\qquad t_-:=\max\{-t,0\}.$$
We shall use $o(1)$ and $O(1)$ in the standard sense to denote quantities that converge to $0$ or are bounded, respectively.

\

\section{Variational formulation}\label{meanfield}

We will devote the first part of this section to introduce the mean-field formulation that we will use to study the problem. This is mostly an adaptation for the singular case of the one proposed in \cite{bls2} for the study of regular case, so we will limit ourselves to collect the main results modifying the notation to the singular case; we recommend to the interested reader the Section 2 of \cite{bls2} for proofs and comments.\\
The last part of this section is devoted to a useful comparison between the direct and the mean-field formulations.

\subsection{Mean field formulation}

\

Following the approach of \cite[Proposition 2.1]{bls2} we can see that the problem \eqref{eqsing} is equivalent to a mean-field type problem.
\begin{proposition}
Problem \eqref{eqsing} is solvable if the following problem has a solution.
\begin{equation}\label{eqmeanfield}
\left\{\begin{array}{ll}
-\Delta u+\frac{4\pi\chi}{|\Sigma|}=2C^2(u)\widetilde Ke^u&\mbox{in }\Sigma\\
\partial_\nu u=2C(u)\widetilde he^{u/2}&\mbox{on }\partial\Sigma\\
C(u)>0\end{array}\right.,
\end{equation}
where $C(u)$ is a root of
\begin{equation}\label{c}
C^2(u)\int_\Sigma\widetilde Ke^u+C(u)\int_{\partial\Sigma}\widetilde he^{u/2}=2\pi\chi
\end{equation}
given by
$$C(u)=\left\{
\begin{array}{ll}
\frac{4\pi\chi}{\sqrt{\left(\int_{\partial\Sigma}\widetilde he^{u/2}\right)^2+8\pi\chi\int_\Sigma\widetilde Ke^u}+\int_{\partial\Sigma}\widetilde he^{u/2}}&\chi>0\\
-\frac{\int_{\partial\Sigma}\widetilde he^{u/2}}{\int_\Sigma\widetilde Ke^u}&\chi=0\\
\frac{4\pi|\chi|}{\sqrt{\left(\int_{\partial\Sigma}\widetilde he^{u/2}\right)^2-8\pi|\chi|\int_\Sigma\widetilde Ke^u}-\int_{\partial\Sigma}\widetilde he^{u/2}}&\chi<0
\end{array}\right..$$
\end{proposition}
In the cases when \eqref{c} has two positive solutions, depending on the values of the two integrals, two possible choices are allowed, getting to equivalent formulations. To be consistent with the global analysis, we choose the branch of solutions which extends the branch of unique solutions when multiplicity fails.\\
The main reason why we introduce this mean-field formulation is that it admits an energy functional whose geometry can be studied in a simpler way than if we used the one given by the direct formulation.
\begin{proposition}\label{meanfieldform}
Solutions to \eqref{eqmeanfield} are critical points of the energy functional $\mathcal J$ defined in \eqref{j} on the space
\begin{equation}\label{hchi}
H_\chi=\left\{u\in\overline H^1(\Sigma):\left\{
\begin{array}{ll}
\int_\Sigma\widetilde Ke^u>-\frac1{8\pi\chi}\left(\int_{\partial\Sigma}\widetilde he^{u/2}\right)_+^2&\chi>0\\
\int_\Sigma\widetilde Ke^u\int_{\partial\Sigma}\widetilde he^{u/2}<0&\chi=0\\
\int_\Sigma\widetilde Ke^u<\frac1{8\pi|\chi|}\left(\int_{\partial\Sigma}\widetilde he^{u/2}\right)_-^2&\chi<0
\end{array}\right\}\right\};
\end{equation}
where 
\begin{equation}\label{fab}
\mathcal F_\chi(A,B)=\left\{\begin{array}{ll}8\pi\chi\left(\log{\left(\sqrt{B^2+8\pi\chi A}+B\right)}+\frac B{\sqrt{B^2+8\pi\chi A}+B}\right)&\chi>0\\
-2\frac{B^2}A&\chi=0\\
8\pi|\chi|\left(-\log{\left(\sqrt{B^2-8\pi|\chi|A}-B\right)}+\frac B{\sqrt{B^2-8\pi|\chi|A}-B}\right)&\chi<0\end{array}\right..
\end{equation}
\end{proposition}
\begin{proof}
It is a straightforward verification as in \cite[Proposition 2.4]{bls2}.
\end{proof}
Clearly, one has to make sure that the energy functional is defined on a nonempty space. To do so, certain considerations with $\widetilde K$ and $\widetilde h$ must be taken into account in order to obtain a solution, arguing as in \cite[Lemma 2.5]{bls2}.
\begin{lemma}
The space $H_\chi$ is non-empty if and only if:
\begin{enumerate}
\item if $\chi>0$, $\widetilde K(x)>0$ for some $x\in\Sigma$ or $\widetilde h(y)>0$ for some $y\in\partial\Sigma$;
\item if $\chi=0$, $\widetilde K(x)\widetilde h(y)<0$ for some $x\in\Sigma$, $y\in\partial\Sigma$;
\item if $\chi<0$, $\widetilde K(x)<0$ for some $x\in\Sigma$ or $\widetilde h(y)<0$ for some $y\in\partial\Sigma$.
\end{enumerate}
Moreover, if such conditions are not satisfied, then problem \eqref{eqsing} is not solvable.
\end{lemma}

\subsection{Relation with the direct formulation}\label{relation}

\

Solution to the problem \eqref{eqsing} can be also seen as critical points to the following energy functional on the whole space $H^1(\Sigma)$:
\begin{equation}\label{i}
\mathcal I(u)=\frac12\int_\Sigma|\nabla u|^2+\frac{4\pi\chi}{|\Sigma|}\int_\Sigma u-2\int_\Sigma\widetilde Ke^u-4\int_{\partial\Sigma}\widetilde he^{u/2}.
\end{equation}
This is the so-called \emph{direct formulation}, which was the approach to regular problem in the papers \cite{lsmr,lsrsr}.\\

In order to stress the relation between the energy functionals \eqref{i} and \eqref{j}, we introduce the following submanifold:
$$\mathcal M:=\left\{u\in\overline H^1\left(\Sigma\right):\int_\Sigma\widetilde Ke^u+\int_{\partial\Sigma}\widetilde he^{u/2}=2\pi\chi\right\},$$
which is just the constraint given by the Gauss-Bonnet formula \eqref{gaussbonnet}. One can easily check that the two functionals coincide, up to a constant, on this manifold: $\mathcal I|_\mathcal M=\mathcal J|_\mathcal M+c$.\\
The same argument works for non geometrical version of both problems, that is
$$\left\{\begin{array}{ll}-\Delta u+\frac\lambda{|\Sigma|}=2C^2(u)\widetilde Ke^u&\mbox{in }\Sigma\\
\partial_\nu u=2C(u)\widetilde he^{u/2}&\mbox{on }\partial\Sigma\\
C(u)=\frac{4\pi\chi}{\sqrt{\left(\int_{\partial\Sigma}\widetilde he^{u/2}\right)^2+8\pi\chi\int_\Sigma\widetilde Ke^u}+\int_{\partial\Sigma}\widetilde he^{u/2}}>0\end{array}\right.,$$
and
$$\left\{\begin{array}{ll}-\Delta u+\frac\lambda{|\Sigma|}=2\widetilde Ke^u&\mbox{in }\Sigma\\
\partial_\nu u=2\widetilde he^{u/2}&\mbox{on }\partial\Sigma\\
\end{array}\right.,$$
whose energy functionals are given respectively by
$$\mathcal J_\lambda(u):=\frac12\int_\Sigma|\nabla u|^2-\mathcal F_{\frac\lambda{4\pi}}\left(\int_\Sigma\widetilde Ke^u,\int_{\partial\Sigma}\widetilde he^{u/2}\right)$$
and
$$\mathcal I_\lambda(u):=\frac12\int_\Sigma|\nabla u|^2+\frac\lambda{|\Sigma|}\int_\Sigma u-2\int_\Sigma\widetilde Ke^u-4\int_{\partial\Sigma}\widetilde he^{u/2}.$$\

The following result highlights an important relation between the energy functionals $\mathcal I_\lambda$ and $\mathcal J_\lambda$ related to the two formulation, as it shows that their Morse index on solutions may differ by only a finite number. This will be used in Section \ref{blowup}, where finiteness of Morse index will be essential to classify concentration phenomena.
\begin{proposition}\label{morseij}
There exists a constant $c_\lambda$ such that $\mathcal I_\lambda$ coincides with $\mathcal J_\lambda+c_\lambda$ on the submanifold
$$\mathcal M_\lambda:=\left\{u\in\overline H^1\left(\Sigma\right):\int_\Sigma\widetilde Ke^u+\int_{\partial\Sigma}\widetilde he^{u/2}=\frac\lambda2\right\}.$$
Moreover, $|\mathrm{ind}(u)-\mathrm{ind^*}(u)|\le1$ where $\mathrm{ind}(u)$ denotes the Morse index of $u$ associated to $\mathcal J_\lambda$:
\begin{equation}\label{index}
\mathrm{ind}(u):=\sup\left\{\dim E:E\subset H^1(\Sigma),\,\mathcal J_\lambda''(u)[\phi,\phi]<0\,\forall\phi\in E\setminus\{0\}\right\};
\end{equation}
and $\mathrm{ind^*}(u)$ is the Morse index associated to the functional $\mathcal I_\lambda$:
\begin{equation}\label{indexi}
\mathrm{ind^*}(u):=\sup\left\{\dim E:E\subset H^1(\Sigma),\,\mathcal I_\lambda''(u)[\phi,\phi]<0\,\forall\phi\in E\setminus\{0\}\right\}.
\end{equation}
\end{proposition}
\begin{proof}
After simple computations one see that
$$\mathcal I_\lambda(u)|_{\mathcal M_\lambda}=\frac12\int_\Sigma|\nabla u|^2-2\int_{\partial\Sigma}\widetilde he^{u/2}-\lambda,$$
and 
\begin{eqnarray*}
\mathcal J_\lambda(u)|_{\mathcal M_\lambda}&=&\frac12\int_\Sigma|\nabla u|^2-2\lambda\left(\log\lambda+\frac{\int_{\partial\Sigma}\widetilde he^{u/2}}\lambda\right)\\
&=&\frac12\int_\Sigma|\nabla u|^2-2\lambda\log\lambda-2\int_{\partial\Sigma}\widetilde he^{u/2},
\end{eqnarray*}
therefore $\mathcal I_\lambda|_{\mathcal M_\lambda}=\mathcal J_\lambda|_{\mathcal M_\lambda}+2\lambda\log\lambda-\lambda$.\\
To deal with the Morse indices, at each point $u\in\mathcal M_\lambda$ we split $H^1(\Sigma)=T_u\mathcal M_\lambda\oplus N_u\mathcal M_\lambda$, where $T_u\mathcal M_\lambda$ is the tangent space of $\mathcal M_\lambda$ at $u$ and $N_u\mathcal M_\lambda$ is its normal space. Since $\mathcal I_\lambda$ and $\mathcal J_\lambda$ coincide on $\mathcal M_\lambda$, up to the constant $c_\lambda$, the tangent spaces will coincide, therefore:
$$\mathcal J''_\lambda(u)=\mathcal J''_\lambda(u)|_{T_u\mathcal M_\lambda}+\mathcal J''_\lambda(u)|_{N_u\mathcal M_\lambda}=\mathcal I''_\lambda(u)|_{T_u\mathcal M_\lambda}+\mathcal J''_\lambda(u)|_{N_u\mathcal M_\lambda}.$$
On the other hand, since $\mathcal M_\lambda$ has co-dimension 2 in $H^1(\Sigma)$ and co-dimension 1 in $\overline H^1(\Sigma)$, then $\mathcal J''_\lambda(u)|_{N_u\mathcal M_\lambda}$ is a two-dimensional quadratic form, which furthermore vanishes in the direction of constants function, due to the invariance of $\mathcal J$. Therefore the number of directions where the quadratic form associated to $\mathcal J_\lambda$ and $\mathcal I_\lambda$ have different sign, that is the difference between the Morse indices, can be at most 1.
\end{proof}\

\section{Energy-minimizing solutions in the case of positive singular Euler characteristic}\label{chipos}

In this section we study the case of $\chi>0$, where we will show that the energy functional \eqref{j} is coercive, hence we will get energy-minimizing solutions. We will divide our discussion into two parts, the first one devoted to the subcritical case and the second one to the critical and supercritical cases, for which we will need an additional symmetric hypothesis.

\subsection{Subcritical case}

\

In the study of the geometry of the energy functional \eqref{j} the logarithmic term plays a crucial role, since the other term in the definition \eqref{fab} of $F(A,B)$ is easily seen to be bounded from above by $1$, namely
$$\frac B{\sqrt{B^2+8\pi\chi A}+B}\le1.$$
This means that the analysis of the energy functional will be based on the study of the logarithmic term in detail. To this purpose we will follow the approach proposed in \cite{bls2}, based on suitable Trudinger-Moser type inequalities. In our case, such inequalities will take account of the singularities in the problem, which are encoded in the Trudinger constant $\tau$ defined in \eqref{tau}.
\begin{proposition}\label{tm}
For any $\varepsilon>0$ there exists $C=C_\varepsilon>0$ such that for any $u\in\overline H^1(\Sigma)$ there holds:
$$\log{\left(\sqrt{\left(\int_{\partial\Sigma}\widetilde he^{u/2}\right)^2+8\pi\chi\int_\Sigma\widetilde Ke^u}+\int_{\partial\Sigma}\widetilde he^{u/2}\right)}\le\frac{1+\varepsilon}{16\tau\pi}\int_\Sigma|\nabla u|^2+C,$$
with $\tau$ as in \eqref{tau}
\end{proposition}

The following two lemmas are essential to get control on domains with conical singularities or corners. The proof of Proposition \ref{tm}, at the end of this sub-section, is an easy consequence of these lemmas.\\
The first lemma is based on a Euclidean singular Trudinger-Moser inequality from \cite{as} and a covering argument.
\begin{lemma}
For any $\varepsilon>0$ there exists a constant $C\in\mathbb R$ such that 
\begin{equation}\label{tmint}
\log\int_\Sigma\widetilde Ke^u\le\frac{1+\varepsilon}{8\tau\pi}\int_\Sigma|\nabla u|^2+C,
\end{equation}
for all $u\in\overline H^1(\Sigma)$.
\end{lemma}
\begin{proof}
For any $x\in\overline\Sigma$ we take $\delta(x)>0$ such that one has isothermal coordinates $\Xi:B_{2\delta(x)}(x)\to B_1(0)\subset\mathbb R^2$ if $x\in\Sigma$ or $\Xi:B_{2\delta(x)}(x)\to B^+_1(0)$ if $x\in\partial\Sigma$. In the former case, for any $u\in H^1_0(B_{2\delta(x)}(x))$ we have $\widetilde u:=u\circ\Xi^{-1}\in H^1_0(B_1(0))$ and
\begin{equation}\label{equiv}
\int_{B_{2\delta(x)}(x)}\widetilde Ke^u\le C\int_{B_1(0)}\left(\widetilde K\circ\Xi^{-1}\right)e^{\widetilde u},\quad\int_{B_{2\delta(x)}(x)}|\nabla u|^2=\int_{B_1(0)}|\nabla\widetilde u|^2,
\end{equation}
latter being a consequence of conformal invariance. On the other hand, if $x\in\partial\Sigma$ and $u\in H^1(B_{2\delta(x)}(x))$ satisfies $u\equiv0$ on $\partial B_{2\delta(x)}(x)\cap\Sigma$, we set
$$\widetilde u(s,t):=\left\{\begin{array}{ll}u\left(\Xi^{-1}(s,t)\right)&\mbox{if }(s,t)\in B^+_1(0)\\u\left(\Xi^{-1}(s,-t)\right)&\mbox{otherwise}\end{array}\right.:$$
as before, we have $\widetilde u\in H^1_0(B_1(0))$, but \eqref{equiv} is replaced with
\begin{equation}\label{equiv2}
\int_{B_{2\delta(x)}(x)}\widetilde Ke^u\le C\int_{B_1(0)}\left(\widetilde K\circ\Xi^{-1}\right)e^{\widetilde u},\quad\int_{B_{2\delta(x)}(x)}|\nabla u|^2=\frac12\int_{B_1(0)}|\nabla\widetilde u|^2.
\end{equation}
It is also not restrictive to assume $\delta(x)$ so small that $B_{2\delta(x)}$ does not include any singular point, except possibly $x$; this allows to apply the singular Trudinger-Moser inequality from Adimurthi and Sandeep (Theorem 2.1 in \cite{as}) and get, if $u\not\equiv0$:
\begin{eqnarray}
\nonumber&&\log\int_{B_1(0)}\left(\widetilde K\circ\Xi^{-1}\right)e^{\widetilde u}\\
\nonumber&\le&\log\int_{B_1(0)}\frac{e^{\widetilde u}}{|x|^{2\alpha(x)_-}}\mathrm dx+C\\
\nonumber&\le&\log\int_{B_1(0)}\frac{e^{\frac{4(1+\alpha(x)_-\pi)\widetilde u^2}{\int_{B_1(0)}|\nabla\widetilde u|^2}}}{|x|^{2\alpha(x)_-}}\mathrm dx+\frac1{16(1+\alpha(x)_-)\pi}\int_{B_1(0)}|\nabla\widetilde u|^2+C\\
\label{tmsing}&\le&\frac1{16\tau\pi}\int_{B_1(0)}|\nabla\widetilde u|^2+C,
\end{eqnarray}
where 
\begin{equation}\label{alpha}
\alpha(x):=\left\{\begin{array}{ll}\alpha_i&x=p_i\\0&x\notin\{p_1,\dots,p_n\}\end{array}\right..
\end{equation}
We now cover $\overline\Sigma$ with the open balls $\{B_{\delta(x)}(x)\}_{x\in\overline\Sigma}$ and, thanks to compactness, $\overline\Sigma\subset\bigcup_{i=1}^NB_{\delta(x_i)}(x_i)$. We then fix $u\in\overline H^1(\Sigma)$ and, up to re-ordering the points, we will have
$$\int_{B_\delta(x_1)}\widetilde Ke^u\ge\frac1N\int_\Sigma\widetilde Ke^u,\qquad\delta:=\delta(x_1).$$
We decompose $u$ as a Fourier series with respect to an orthonormal frame of eigenfunctions $\varphi_n$ for $-\Delta$ on $\overline H^1(\Sigma)$ with associated eigenvalues $0<\varsigma_n\underset{n\to\infty}\nearrow\infty$ and we split
$$u=\sum_{n=1}^\infty u_n\varphi_n=\underbrace{\sum_{n=1}^Nu_n\varphi_n}_{=:v}+\underbrace{\sum_{n=N}^\infty u_n\varphi_n}_{=:w},$$
with $N=N_{\varepsilon,\delta}$ to be chosen later. Since $v$ belongs to an $N$-dimensional space, we get
\begin{equation}\label{v}
\|v\|_{L^\infty(\Sigma)}\le C\sqrt{\int_\Sigma|\nabla v|^2}\le\frac{C^2}{4\varepsilon}+\varepsilon\int_\Sigma|\nabla v|^2\le\frac{C^2}{4\varepsilon}+\varepsilon\int_\Sigma|\nabla u|^2,
\end{equation}
with $C$ depending on $N$, hence only on $\varepsilon,\delta$; on the other hand, by the choice of $N$ we get
$$\int_\Sigma w^2\le\frac1{\varsigma_{N+1}}\int_\Sigma|\nabla w|^2\le\frac1{\varsigma_{N+1}}\int_\Sigma|\nabla u|^2.$$
Take now a smooth cut-off $\zeta$ such that
\begin{equation}\label{zeta}
\zeta\in\mathcal C^\infty_0(B_{2\delta}(x)),\qquad\zeta\equiv1\mbox{ on }B_\delta(x),\qquad|\nabla\zeta|\le C_\delta,
\end{equation}
so that $\zeta w$ vanishes on $\partial B_\delta(x_1)$ and we can apply \eqref{equiv}, \eqref{equiv2} and \eqref{tmsing} to $\widetilde{\zeta w}\in H^1_0(B_1(0))$; we either get
$$\log\int_{B_{2\delta}(x_1)}\widetilde Ke^{\zeta w}\le\frac1{16\tau\pi}\int_{B_{2\delta}(x_1)}|\nabla(\zeta w)|^2+C$$
or
\begin{equation}\label{zetaw}
\log\int_{B_{2\delta}(x_1)}\widetilde Ke^{\zeta w}\le\frac1{8\tau\pi}\int_{B_{2\delta}(x_1)}|\nabla(\zeta w)|^2+C,
\end{equation}
the latter clearly holding true in both cases.\\
Therefore, by choosing
$$N=N_{\varepsilon,\delta}:=\max\left\{n\in\mathbb N:\,\varsigma_n<\frac{1+\varepsilon}{\varepsilon^2}C_\delta\right\}$$
with $C_\delta$ as in \eqref{zeta}, we get:
\begin{eqnarray*}
\int_{B_{2\delta}(x_1)}|\nabla(\zeta w)|^2&=&\int_\Sigma\left(\zeta^2|\nabla w|^2+2\zeta w\nabla w\cdot\nabla\zeta+w^2|\nabla\zeta|^2\right)\\
&\le&\int_\Sigma\left((1+\varepsilon)\zeta^2|\nabla w|^2+\left(1+\frac1\varepsilon\right)|\nabla\zeta|^2w^2\right)\\
&\le&(1+\varepsilon)\int_\Sigma|\nabla w|^2+\left(1+\frac1\varepsilon\right)C_\delta\int_\Sigma w^2\\
&\le&(1+\varepsilon)\int_\Sigma|\nabla w|^2+\varepsilon\int_\Sigma|\nabla w|^2\\
&\le&(1+2\varepsilon)\int_\Sigma|\nabla u|^2.
\end{eqnarray*}
Together with \eqref{v} and \eqref{zetaw}, this gets to the conclusion:
\begin{eqnarray*}
\log\int_\Sigma\widetilde Ke^u&\le&\log\int_{B_\delta(x_1)}\widetilde Ke^u+\log N\\
&\le&\log\int_{B_{2\delta}(x_1)}\widetilde Ke^{\zeta w}+\|v\|_{L^\infty(\Sigma)}+\log N\\
&\le&\frac1{8\tau\pi}\int_{B_{2\delta}(x_1)}|\nabla(\zeta w)|^2+\varepsilon\int_\Sigma|\nabla u|^2+C\\
&\le&\frac{1+2\varepsilon}{8\tau\pi}\int_\Sigma|\nabla u|^2+\varepsilon\int_\Sigma|\nabla u|^2+C\\
&=&\frac{1+\varepsilon'}{8\tau\pi}\int_\Sigma|\nabla u|^2+C.
\end{eqnarray*}
\end{proof}

To deal with boundary integral term we have the following result, in the same spirit as \cite{tr,ll}.
\begin{lemma}\label{tmbdry}
For any $\varepsilon>0$ there exists a constant $C=C_\varepsilon\in\mathbb R$ such that 
$$\log\left|\int_{\partial\Sigma}\widetilde he^{u/2}\right|\le\frac{1+\varepsilon}{16\tau\pi}\int_\Sigma|\nabla u|^2+C,$$
for all $u\in\overline H^1(\Sigma)$.
\end{lemma}
\begin{proof}
If $\tau=1$, then $\widetilde h$ is bounded, therefore we just suffice to apply Corollary 2.6 from \cite{cr}, which says that
\begin{equation}\label{liliu}
\log\int_{\partial\Sigma}e^{u/2}\le\frac1{16\pi}\int_\Sigma|\nabla u|^2+C.
\end{equation}
If $\tau<1$, then $h\in L^p(\partial\Sigma)$ for $p=\frac{1+\varepsilon}{1+\varepsilon-\tau}<\frac1{1-\tau}$, therefore applying H\"older's inequality and then \eqref{liliu} to $\frac{1+\varepsilon}\tau u$ we get
\begin{eqnarray*}
\log\left|\int_{\partial\Sigma}\widetilde he^{u/2}\right|&\le&\frac{1+\varepsilon-\tau}{1+\varepsilon}\log\int_{\partial\Sigma}\left|\widetilde h\right|^\frac{1+\varepsilon}{1+\varepsilon-\tau}+\frac\tau{1+\varepsilon}\log\int_{\partial\Sigma}e^{(1+\varepsilon)/(2\tau)u}\\
&\le&\frac{1+\varepsilon}{16\tau\pi}\int_\Sigma|\nabla u|^2+C.
\end{eqnarray*}
\end{proof}

\begin{proof}[Proof of Proposition \ref{tm}]
We consider the elementary inequality
\begin{equation}\label{ineq}
\sqrt{B^2+8\pi A}+B\le\max\left\{\sqrt A,|B|\right\}\left(\sqrt{1+8\pi\chi}+1\right),\quad\mbox{for all }A>0,B\in\mathbb R,
\end{equation}
with
$$A=\int_\Sigma\widetilde Ke^u,\qquad B=\int_{\partial\Sigma}\widetilde he^{u/2}.$$
Therefore
$$\log{\left(\sqrt{B^2+8\pi\chi A}+B\right)}\le\max\left\{\frac12\log A,\log|B|\right\}+C\le\left(\frac1{16\tau\pi}+\varepsilon\right)\int_\Sigma|\nabla u|^2+C.$$
\end{proof}
In the case when $\alpha_i$'s and $\beta_j$'s are all non-negative, that is when $\widetilde h,\widetilde K\in L^\infty$ and $\tau=1$, Lemmas \ref{tmint} and \ref{tmbdry} are well-known to hold true, even with $\varepsilon=0$, respectively from \cite{cy} (Corollary 2.5) and \cite{cr} (Corollary 2.6). We believe that both lemmas, hence Proposition \ref{tm}, still hold true with $\varepsilon=0$, although the proof should be much more involved.

\subsection{Critical and supercritical cases with symmetries}\label{min}

\

In the critical and supercritical cases, using suitable test functions one can see that the energy $\mathcal J$ is not coercive on $\overline H^1(\Sigma)$. To bypass this issue we exploit the symmetry assumption \eqref{sym}, which gives and improved version of Proposition \ref{tm}.\\
We start by giving some \emph{localized} versions of the Trudinger-Moser inequalities, see also \cite{bjwy,cl2,cr}.

\begin{proposition}\label{tmloc}
For every $\delta,\varepsilon>0$ there exists a constant $C=C_{\varepsilon,\delta}>0$ such that for any open $\Omega\subset\Sigma$ and $u\in\overline H^1(\Sigma)$ one has
$$8\tau\pi\log\int_{\Omega}\widetilde Ke^u\le\int_{\Omega^\delta}|\nabla u|^2+\varepsilon\int_\Sigma|\nabla u|^2+C.$$
If $\Omega^\delta\cap\partial\Sigma=\emptyset$ for some $\delta>0$ and $\Omega^\delta$ contains no conical singularity, then
$$16\pi\log\int_{\Omega}\widetilde Ke^u\le\int_{\Omega^\delta}|\nabla u|^2+\varepsilon\int_\Sigma|\nabla u|^2+C:$$
if $\Omega^\delta\cap\partial\Sigma=\emptyset$ and $\Omega^\delta\cap\{p_1,\dots,p_n\}=p_i$, then
$$16\pi\min\{1,1+\alpha_i\}\log\int_{\Omega}\widetilde Ke^u\le\int_{\Omega^\delta}|\nabla u|^2+\varepsilon\int_\Sigma|\nabla u|^2+C.$$
\end{proposition}
The details of the proof can be found in \cite[Proposition 2.2]{lsr} in a regular setting, but in the singular case a similar argument works. The main idea is to apply \eqref{tmint} to the function $u$ multiplied by a cut-off function on $\Omega$. An immediate consequence of the above result is the following corollary, where a concentration of the conformal volume gives a global control of the volume in terms of the Dirichlet energy.

\begin{corollary}\label{tmimprint}
Assume $\Omega_1,\dots,\Omega_k\subset\Sigma$ are open domains satisfying, for some $\delta>0$,
$$\Omega_i^\delta\cap\Omega_j^\delta=\emptyset\quad\forall i,j=1,\dots,k,i\ne j.$$
\begin{equation}\label{mass}
\frac{\int_{\Omega_i}\widetilde Ke^u}{\int_\Sigma\widetilde Ke^u}\ge\delta\quad\forall i=1,\dots,k.
\end{equation}
Then, for every $\varepsilon>0$ there exists a constant $C=C_{\varepsilon,\delta}>0$ such that for any $u\in\overline H^1(\Sigma)$ one has
$$\log\int_\Sigma\widetilde Ke^u\le\frac{1+\varepsilon}{8k\tau\pi}\int_\Sigma|\nabla u|^2+C,$$
\end{corollary}
\begin{proof}
Applying Proposition \ref{tmloc} to each $\Omega_i$ one has
$$8\tau\pi\log\int_\Sigma\widetilde Ke^u\le8\tau\pi\log\left(\frac1\delta\int_{\Omega_i}\widetilde Ke^u\right)\le\int_{\Omega_i^\delta}|\nabla u|^2+\varepsilon\int_\Sigma|\nabla u|^2+C_{\varepsilon,\delta}.$$
Since the $\Omega_i^\delta$'s are pairwise disjoint, then summing all of them we get:
$$8k\tau\pi\log\int_\Sigma\widetilde Ke^u\\
\le\int_{\bigcup_{i=0}^{k-1}\Omega_i^\delta}|\nabla u|^2+k\varepsilon\int_\Sigma|\nabla u|^2+C_{\varepsilon,\delta}\le(1+k\varepsilon)\int_\Sigma|\nabla u|^2+C.$$
\end{proof}
The same ideas can be used to give a localised version of the Lemma \ref{tmbdry} (see also Proposition 2.10 of \cite{cr}).
\begin{proposition}
For every $\varepsilon>0$ there exists a constant $C=C_\varepsilon$ such that for any $\Theta\subset\partial\Sigma$ and $u\in\overline H^1(\Sigma)$ one has
$$16\tau\pi\log\int_\Theta\widetilde he^{u/2}\le\int_{\Theta^\delta}|\nabla u|^2+\varepsilon\int_\Sigma|\nabla u|^2+C.$$
\end{proposition}
\begin{corollary}\label{tmimprbdry}
Assume $\Theta_1,\dots,\Theta_k\subset\partial\Sigma$ are open domains satisfying, for some $\delta>0$,
$$\Theta_i^\delta\cap\Theta_j^\delta=\emptyset\qquad\forall i,j=1,\dots,k,i\ne j,$$
\begin{equation}\label{massbdry}
\frac{\int_{\Theta_j}\widetilde he^{u/2}}{\int_{\partial\Sigma}\widetilde he^{u/2}}\ge\delta\quad\forall j=1,\dots,k.
\end{equation}
Then, for every $\varepsilon>0$ there exists a constant $C=C_\varepsilon>0$ such that, for any $u\in\overline H^1(\Sigma)$,
$$\log\int_\Sigma\widetilde he^{u/2}\le\frac{1+\varepsilon}{16k\tau\pi}\int_\Sigma|\nabla u|^2+C.$$
\end{corollary}

From now on, we assume that $K$ and $h$ are invariant under the action of an isometry group $\mathcal G$ with no fixed points on the boundary (as in \eqref{sym}). This symmetry assumption allows us to improve the constant in the inequality stated in Proposition \ref{tm} by employing the refined Trudinger-Moser inequality from Corollaries \ref{tmimprint} and \eqref{tmimprbdry}, which in turn yield coercivity. The argument is analogous to that of Proposition 2.12 in \cite{cr} (see also \cite{bls2,mo}).\\

Moreover, it is easy to see that if $u\in\widetilde H_{\mathcal G}$ does not vanish on some region of $\Sigma$, then the assumptions in \eqref{mass} and \eqref{massbdry} are satisfied. Consequently, the constant in the Trudinger-Moser inequality can be improved by 
$$k=\min\left\{\#\mathcal Gx:x\in\Sigma\setminus \mathrm{Fix}(\mathcal G)\right\}.$$
In our case, the only additional issue to consider is whether there are fixed points in the interior of $\Sigma$ or not.

\begin{proposition}\label{tmimpr}
For any $\varepsilon>0$ there exists a constant $C=C_\varepsilon>0$ such that for any $u\in\widetilde H_{\mathcal G}$, the following holds:
$$\log{\left(\sqrt{\left(\int_{\partial\Sigma}\widetilde he^{u/2}\right)^2+8\pi\chi\int_\Sigma\widetilde Ke^u}+\int_{\partial\Sigma}\widetilde he^{u/2}\right)}\le\frac{1+\varepsilon}{16\pi\sigma}\int_\Sigma|\nabla u|^2+C,$$
where
$$\sigma=\left\{\begin{array}{ll}k\tau&\mbox{if }\mathrm{Fix}(\mathcal G)=\emptyset\\\min\left\{2,2+2\min_{i:\,p_i\in\mathrm{Fix}(\mathcal G)}\alpha_i,k\tau\right\}&\mbox{if }\mathrm{Fix}(\mathcal G)\ne\emptyset,
\end{array}\right.,$$
\end{proposition}
and $k$ is as in \eqref{k}.
\begin{proof}
We first establish an estimate for the boundary term and then for the interior term.\\
For any $x\in\partial\Sigma$, we choose elements $g_{1,x},\dots,g_{l,x}\in\mathcal G$ that do not stabilize $x$, and select a small $\delta(x)>0$ such that the points $g_{i,x}x$ are at least $2\delta(x)$ apart, i.e.,
\begin{equation}\label{disjoint} 
B_{2\delta(x)}(g_{i,x}x)\cap B_{2\delta(x)}(g_{j,x}x)=\emptyset,\quad\forall\,i\ne j.
\end{equation}\\
We then cover $\partial\Sigma$ with the open balls $\{B_{\delta(x)}(x)\}_{x\in\partial\Sigma}$ so that, by compactness,
$$\partial\Sigma\subset\bigcup_{i=1}^N B_{\delta(x_i)}(x_i),$$
for some $x_i,\dots,x_N$. Without loss of generality, assume that
$$\frac{\int_{B_{\delta(x_1)}\cap\partial\Sigma}\widetilde he^{u/2}}{\int_{\partial\Sigma}\widetilde he^{u/2}}\ge\frac1N,$$
where $\delta=\delta(x_1)$. Since $u\in\widetilde H_{\mathcal G}$, for every $g_i=g_{i,x_1}$ we have
$$\frac{\int_{B_{\delta}(g_i x_1)\cap\partial\Sigma}\widetilde he^{u/2}}{\int_{\partial\Sigma}\widetilde he^{u/2}}=\frac{\int_{B_{\delta}(x_1)\cap\partial\Sigma}\widetilde he^{u/2}}{\int_{\partial\Sigma}\widetilde he^{u/2}}\ge\frac1N.$$
Moreover, by our choice in \eqref{disjoint} the balls $B_{\delta}(g_ix_1)$ are at a mutual distance of at least $2\delta$. Thus, by applying Corollary \ref{tmimprbdry} we obtain
\begin{equation}\label{tmsymbdry}
\log\int_{\partial\Sigma}\widetilde he^{u/2} \le \frac{1+\varepsilon}{16k\tau\pi}\int_\Sigma|\nabla u|^2+C.
\end{equation}

Next, consider the interior estimate. We distinguish two cases.
\begin{itemize}
\item[Case 1: $\mathrm{Fix}(\mathcal G)=\emptyset$] Similarly to before, by a covering argument, up to choosing a smaller $\delta>0$ there exist points $x_1,g_1x_1,\dots,g_lx_1$ such that
$$B_{2\delta}(g_ix_1)\cap B_{2\delta}(g_j x_1)=\emptyset,\quad\forall\,i\ne j,$$
and
$$\frac{\int_{B_\delta(g_ix_1)\cap\Sigma}\widetilde Ke^u}{\int_\Sigma\widetilde Ke^u}=\frac{\int_{B_\delta(x_1)\cap\Sigma}\widetilde Ke^u}{\int_\Sigma\widetilde Ke^u}\ge\frac1N,\quad\forall\,i=1,\dots,l.$$
Then, applying Corollary \ref{tmimprint}, we deduce
$$\log\int_\Sigma\widetilde Ke^u \le \frac{1+\varepsilon}{8k\tau\pi}\int_\Sigma|\nabla u|^2+C.$$
Combining this with the boundary estimate and using \eqref{ineq}, we obtain
\begin{equation}\label{tmsym}
\log\left(\sqrt{\left(\int_{\partial\Sigma}\widetilde he^{u/2}\right)^2+8\pi\chi\int_\Sigma\widetilde Ke^u}+\int_{\partial\Sigma}\widetilde he^{u/2}\right)\le\frac{1+\varepsilon}{16k\tau\pi}\int_\Sigma|\nabla u|^2+C.
\end{equation}

\item[Case 2: $\mathrm{Fix}(\mathcal G)\ne\emptyset$] Since $\mathrm{Fix}(\mathcal G)$ is compact, we may fix a small $\delta>0$ such that 
$$d(\partial\Sigma,\mathrm{Fix}(\mathcal G))\ge2\delta$$
and, for any $x\in\mathrm{Fix}(\mathcal G)$, $B_\delta(x)\subset\Sigma$ and it contains no conical point except possibly $x$. Then, either
\begin{equation}\label{out}
\frac{\int_{\Sigma\setminus B_\delta(x)}\widetilde Ke^u}{\int_\Sigma\widetilde Ke^u}\ge\frac12
\end{equation}
or
\begin{equation}\label{in}
\frac{\int_{B_\delta(x)}\widetilde Ke^u}{\int_\Sigma\widetilde Ke^u}\ge\frac12.
\end{equation}
In case \eqref{out}, by a similar covering argument as in Case 1, we obtain
\begin{equation}\label{int}
\log\int_\Sigma\widetilde Ke^u\le\frac{1+\varepsilon}{8k\tau\pi}\int_\Sigma|\nabla u|^2+C.
\end{equation}
Up to taking a smaller $\delta$, there exist $N\in\mathbb N,x_1,g_1x_1,\dots,g_lx_1$ such that
$$B_{2\delta}\left(g_ix_1\right)\cap B_{2\delta}\left(g_jx_1\right)=\emptyset\quad\forall\,i,j=1,\dots,l,\,i\ne j;$$
$$\frac{\int_{B_\delta(g_ix_1)\cap\left(\Sigma\setminus B_\delta(x)\right)}\widetilde Ke^{u/2}}{\int_\Sigma\widetilde Ke^{u/2}}\ge\frac12\frac{\int_{B_\delta(g_ix_1)\cap\left(\Sigma\setminus B_\delta(x)\right)}\widetilde Ke^{u/2}}{\int_{\Sigma\setminus B_\delta(x)}\widetilde Ke^{u/2}}\ge\frac1{2N},\,\forall\quad i=1,\dots,l.$$
From Corollary \ref{tmimprint} one gets \eqref{int}, hence, together with \eqref{tmsymbdry} and \eqref{ineq}, we get again \eqref{tmsym}.\\
If instead \eqref{in} holds true, then by applying Proposition \ref{tmloc}, with the coefficient $\alpha(p)$ defined in \eqref{alpha}, we have
\begin{eqnarray*}
16\pi\min\{1,1+\alpha(p)\}\log\int_\Sigma\widetilde Ke^u&\le&16\pi\min\{1,1+\alpha(p)\}\log\left(\frac1\delta\int_{B_\delta(0)}\widetilde Ke^u\right)\\
&\le&(1+\varepsilon)\int_\Sigma|\nabla u|^2+C.
\end{eqnarray*}
In either case, we deduce
$$\log\int_\Sigma\widetilde Ke^u\le\frac{1+\varepsilon}{8\pi\min\left\{2,2+2\min_{i:\,p_i\in\mathrm{Fix}(\mathcal G)}\alpha_i,k\tau\right\}}\int_\Sigma|\nabla u|^2+C.$$
Finally, combining this with the boundary estimate and using \eqref{ineq} yields
$$\log\left(\sqrt{\left(\int_{\partial\Sigma}\widetilde he^{u/2}\right)^2+8\pi\chi\int_\Sigma\widetilde Ke^u}+\int_{\partial\Sigma}\widetilde he^{u/2}\right)\le\frac{1+\varepsilon}{16\pi\sigma}\int_\Sigma|\nabla u|^2+C,$$
with
$$\sigma=\left\{\begin{array}{ll}k\tau&\mbox{if }\mathrm{Fix}(\mathcal G)=\emptyset\\\min\left\{2,2+2\min_{i:\,p_i\in\mathrm{Fix}(\mathcal G)}\alpha_i,k\tau\right\},&\mbox{if }\mathrm{Fix}(\mathcal G)\ne\emptyset\end{array}\right..$$
This completes the proof.
\end{itemize}
\end{proof}

We are now in position to prove Theorem \ref{chi>0}.
\begin{proof}[Proof of Theorem \ref{chi>0}]
First of all, since $K\gneqq0$, then $\int_\Sigma\widetilde Ke^u>0$ for any $u$, so the domain $H_\chi$ of the energy functional $\mathcal J$ coincides with the whole $\overline H^1(\Sigma)$ (see \eqref{hchi}) and existence of minimizing solutions will follow by showing coercivity in $\overline H^1(\Sigma)$ in the subcritical case and coercivity in $\widetilde H_{\mathcal G}$ in the critical and subcritical case.
\begin{itemize}
\item{Subcritical case:}\\
From Proposition \ref{tm}, for any $\varepsilon>0$ there exists $C>0$ such that for any $u\in\overline H^1(\Sigma)$, we have:
$$\mathcal F_\chi\left(\int_\Sigma\widetilde Ke^u,\int_{\partial\Sigma}\widetilde he^{u/2}\right)\le\frac12\left(\frac\chi\tau+\varepsilon\right)\int_\Sigma|\nabla u|^2+C,$$
implying:
$$\mathcal J(u)\ge\frac12\left(1-\frac\chi\tau-\varepsilon\right)\int_\Sigma|\nabla u|^2-C.$$
Since we are in the subcritical case, we can take $\varepsilon<1-\frac\chi\tau$ so that $\mathcal J$ is coercive, hence ensuring the existence of minimizers.

\item{Critical and supercritical cases:}\\
In the case $\mathrm{Fix}(\mathcal G)=\emptyset$, Proposition \ref{tmimpr} states that for any $\varepsilon>0$ one has
$$\mathcal F_\chi\left(\int_\Sigma\widetilde Ke^u,\int_{\partial\Sigma}\widetilde he^{u/2}\right)\le\frac12\left(\frac\chi{k\tau}+\varepsilon\right)\int_\Sigma|\nabla u|^2+C,$$
for any $u\in\widetilde H_{\mathcal G}$, hence
$$\mathcal J(u)\ge\frac12\left(1-\frac\chi{k\tau}-\varepsilon\right)\int_\Sigma|\nabla u|^2-C.$$
Taking $\varepsilon<1-\frac\chi{k\tau}$, we ensure coercivity and existence of minimizers.\\
The cases when $\mathcal G$ has fixed points on $\Sigma$ are dealt with similarly.
\end{itemize}
\end{proof}\

\section{Blow-up analysis}\label{blowup}

This section will be focused on proving the following result concerning blow-up analysis of sequence of solutions to problem \eqref{eqsing} and some other problems which approximate \eqref{eqsing}.\\
The main result of this Section, which will be proved in the end of the section, is the following.
\begin{theorem}\label{blowupanal}
Let $u_n$ be a sequence of solutions to the problem:
$$\left\{\begin{array}{ll}
-\Delta u_n+\frac{\lambda_n}{|\Sigma|}=2\widetilde K_ne^{u_n}&\mbox{in }\Sigma\\
\partial_\nu u_n=2\widetilde h_ne^{u_n/2}&\mbox{on }\partial\Sigma
\end{array}\right.,$$
where $\lambda_n\to\lambda\in\mathbb R$, $\widetilde K_n\to\widetilde K>0$, and $\widetilde h_n\to\widetilde h$ in the $\mathcal C^{0,\gamma}$ sense, with $\widetilde K$ and $\widetilde h$ as in \eqref{kh}. The singular set is defined as:
\begin{equation}\label{s}
\mathcal S:=\left\{p\in\overline\Sigma:\exists x_n\to p\mbox{ as }n\to\infty\mbox{ such that }u_n(x_n)\to\infty\right\}.
\end{equation}
The following alternatives hold:
\begin{enumerate}
\item If $\int_\Sigma\widetilde K_ne^{u_n}$ is uniformly bounded, then, up to a subsequence, either:
\begin{enumerate}
\item $u_n$ is uniformly bounded in $L^\infty(\Sigma)$, or;
\item $\mathcal S$ is finite and non-empty, such that
\begin{eqnarray}
\label{deltas}2\widetilde K_ne^{u_n}&\underset{n\to\infty}\rightharpoonup&\sum_{p\in\mathcal S\cap\Sigma}8\pi(1+\alpha(p))\delta_p+\sum_{p\in\mathcal S\cap\partial\Sigma}\gamma(p)\delta_p,\\
\nonumber2\widetilde h_ne^{u_n/2}&\underset{n\to\infty}\rightharpoonup&\sum_{p\in\mathcal S\cap\partial\Sigma}(4\pi(1+\beta(p))-\gamma(p))\delta_p,
\end{eqnarray}
for some $\gamma(p)\in\mathbb R$, $\alpha(p)$ as in \eqref{alpha} and
\begin{equation}\label{beta}
\beta(p):=\left\{\begin{array}{ll}\beta_j&p=q_j\\0&p\notin\{q_1,\dots,q_m\}\end{array}\right..
\end{equation}
\end{enumerate}
\item If $\int_\Sigma\widetilde Ke^{u_n}$ is unbounded, then $\mathrm{ind}(u_n)\underset{n\to\infty}\to\infty.$
\end{enumerate}
\end{theorem}
Here, $\mathrm{ind}(u_n)$ represents the Morse index of the perturbed energy functional
$$\mathcal J_n(u):=\frac12\int_\Sigma|\nabla u|^2-\mathcal F_{\frac{\lambda_n}{4\pi\chi}}\left(\int_\Sigma\widetilde K_ne^u,\int_{\partial\Sigma}\widetilde h_ne^{u/2}\right)$$
on the functions $u_n$, defined as in \eqref{index}.
This is essentially the number of negative eigenvalues of $\mathcal J_n''(u_n)$, counted with multiplicity.\\
Blow-up with infinite mass is a new major issue with respect to classical mean field problems: in fact, Gauss-Bonnet theorem \eqref{gaussbonnet} gives a constraint on the sum of the two nonlinear integrals but, if $K,h$ are sign-changing or have different sign, they may both diverge to $\infty$. The use of the Morse index allows to rule out this phenomenon, as in \cite{lsmr}, since solutions obtained with a finite-dimensional min-max scheme have uniformly bounded Morse index.\\

An easy consequence of Theorem \ref{blowupanal} is that alternative (1.b) occurs only if the parameter $\lambda$ belongs to a discrete set.
\begin{corollary}\label{compa}
Let $\Gamma$ be as in \eqref{gamma}. If $\lambda\notin\Gamma$ and $\mathrm{ind}(u_n)$ is uniformly bounded, then $u_n$ is uniformly bounded in $L^\infty(\Sigma)$.
\end{corollary}

To perform the blow-up, we first recall the key ideas from the classical works \cite{bm,ls}, which were originally formulated for the regular case in closed surfaces and subsequently extended to the singular case in \cite{ta}. These works indicate that around any isolated concentration point $p\in\Sigma$, suitably rescaled solutions converge to solution to a limit problem on $\mathbb R^2$, whose solutions have been classified, provided the mass is bounded.
The limit problem is given by
$$\left\{\begin{array}{ll}-\Delta v=2K_0|x|^{2\alpha}e^v&\mbox{in }\mathbb R^2\\\int_{\mathbb R^2}|x|^{2\alpha}e^v<\infty
\end{array}\right.,$$
where $K_0>0$ and $\alpha=\alpha(p)$ as in \eqref{alpha}. Solutions have been classified by Chen and Li in \cite{cl1} for $\alpha=0$ and by Prajapat and Tarantello in \cite{pt} for $\alpha\ne0$.\\
When $\partial\Sigma\ne\emptyset$ and blow-up occurs at the boundary, the corresponding limiting problem has non-linear Neumann boundary condition in the half-space:
$$\left\{\begin{array}{ll}
-\Delta v=2 K_0|x|^{2\beta}e^v&\mbox{in }\mathbb R^2_+\\\partial_\nu v=2h_0|x|^\beta e^{v/2}&\mbox{on }\partial\mathbb R^2_+\\\int_{\mathbb R^2_+}|x|^{2\beta}e^v+\int_{\partial\mathbb R^2_+}|x|^\beta e^{v/2}<\infty\end{array}\right.,$$
where $K_0>0$ and $\beta=\beta(p)$ as in \eqref{beta}. The classification result for such problem has been obtained by Zhang in \cite{z} for $\beta=0$ and by Jost, Wang and Zhou in \cite{jwz}.\\
In the study of problem \eqref{eqsing} we may not have bounded mass, since both integrals in \eqref{gaussbonnet} may be infinite if $\widetilde K$ and $\widetilde h$ have different sign; in the assumptions of Theorem \ref{minmax} this could happen if $\widetilde h$ is sign-changing or everywhere negative. This is a new major issue, which affects massively the blow-up analysis, as in the works \cite{lsmr,lsrsr} which do not assume a bounded mass condition.\\
We tackle this issue by means of the Morse index, namely we show that any solution to limiting problems with infinite mass must have infinite Morse index. The first part of this section is devoted to the Morse index of entire functions, whereas in the second part we conclude the blow-up analysis and prove Theorem \ref{blowupanal}.

\subsection{Solutions to the limit problems and their Morse Index}\label{limitprob}

\

In this subsection, we study the Morse index of the solutions to two limit problems. One problem is defined on the whole space and the other on the half-space, that is, respectively,
\begin{equation}\label{eqplane}
-\Delta v=2K_0|x|^{2\alpha}e^v\quad\mbox{in }\mathbb R^2,
\end{equation}
and
\begin{equation}\label{eqhalfplane}
\left\{\begin{array}{ll}-\Delta v=2 K_0|x|^{2\beta}e^v&\mbox{in }\mathbb R^2_+\\\partial_\nu v=2h_0|x|^\beta e^{v/2}&\mbox{on }\partial\mathbb R^2_+
\end{array}\right..
\end{equation}

Solutions to \eqref{eqplane} and \eqref{eqhalfplane} without mass assumptions have been classified, respectively, in \cite{l} and \cite{gjm}:
\begin{theorem}[\cite{l}]
Any solution $v$ to \eqref{eqplane} is in the form
$$v=\log\left(\frac{4|g'|^2}{\left(1+K_0|g|^2\right)^2}\right)-2\alpha\log|x|,$$
where $g:\mathbb C\setminus\{0\}\to\overline{\mathbb C}$ is a locally univalent meromorphic function.
\end{theorem}
\begin{theorem}[Theorem 2, \cite{gjm}]\label{solhalfplane}
Any solution $v$ to \eqref{eqhalfplane} is in the form
$$v=\log\left(\frac{4|g'|^2}{\left(1+K_0|g|^2\right)^2}\right)-2\beta\log|x|,$$
where $g:\overline{\mathbb C}\setminus\{0\}\to\overline{\mathbb C}$ is a locally univalent meromorphic function.
\end{theorem}
The definition of the Morse index of a solution $v$ of problem \eqref{eqplane} is as follows:
$$\mathrm{ind}(v):=\sup\left\{\dim E:E\subset\mathcal C^\infty_0\left(\mathbb R^2\right),\,\mathcal Q_v[\phi,\phi]<0\,\forall\phi\in E\setminus\{0\}\right\},$$ where $\mathcal Q_v$ is defined as:
$$\mathcal Q_v[\phi,\phi]=\int_{\mathbb R^2}\left(|\nabla\phi|^2-2K_0|x|^{2\alpha}e^v\phi^2\right).$$
Similarly, the Morse index of a solution $v$ of problem \eqref{eqhalfplane} is defined as:
$$\mathrm{ind}(v):=\sup\left\{\dim E: E\subset\mathcal C^\infty_0\left(\overline{\mathbb R^2_+}\right),\,\mathcal Q_v[\phi,\phi]<0\,\forall\phi\in E\setminus\{0\}\right\},$$ where $\mathcal Q_v$ is dfined as:
$$\mathcal Q_v[\phi,\phi]=\int_{\mathbb R^2_+}\left(|\nabla\phi|^2-2K_0|x|^{2\beta}e^v\phi^2\right)-\int_{\partial\mathbb R^2_+}h_0|x|^\beta e^{v/2}\phi^2.$$
The main result of this subsection is the following:
\begin{theorem}\label{morseindex}\
\begin{enumerate}
\item The Morse index of any solution to \eqref{eqplane} satisfies $\mathrm{ind}(v)\ge1$. Moreover,
$$\mbox{if}\qquad\int_{\mathbb R^2}|x|^{2\alpha}e^v=\infty\qquad\mbox{then}\qquad\mathrm{ind}(v)=\infty.$$
\item The Morse index of any solution to \eqref{eqhalfplane} satisfies $\mathrm{ind}(v)\ge1$. Moreover,
$$\mbox{if}\qquad\int_{\mathbb R^2_+}|x|^{2\beta}e^v=\infty\qquad\mbox{then}\qquad\mathrm{ind}(v)=\infty.$$
\end{enumerate}
\end{theorem}

\begin{proof}\
\begin{enumerate}
\item For the first statement it suffices to find $\phi\in\mathcal C^\infty_0\left(\mathbb R^2\right)$ such that $\mathcal Q_v[\phi,\phi]<0$. We fix $R>1$ and define $\phi:=\eta\left(1-\frac{\log|x|}{\log R}\right)$, with $\eta\in\mathcal C^\infty_0(\mathbb R)$ such that $\eta_{(-\infty,0]}\equiv0,\eta_{[1,\infty)}\equiv1$. Therefore, $\phi$ verifies
$$\left\{\begin{array}{ll}\phi\equiv1&\mbox{ in }B_1(0)\\
\phi\ge0&\mbox{ in }\mathbb R^2\\
\int_{\mathbb R^2}|\nabla\phi|^2\le\frac{C_0}{\log R}&,\end{array}\right.$$
hence
$$\mathcal Q_v[\phi,\phi]\le\int_{\mathbb R^2}|\nabla\phi|^2-\int_{B_1(0)}2K_0|x|^{2\alpha}e^v\le\frac{C_0}{\log R}-2K_0\int_{B_1(0)}|x|^{2\alpha}e^v,$$
which is negative if $R$ is large enough.\\
To prove the second statement, we will show that $v$ is unstable outside any compact set, that is: given $M_0>0$ we will find $\psi\in\mathcal C^\infty_0\left(\mathbb R^2\setminus\overline{B_{M_0}(0)}\right)$ such that $\mathcal Q_v[\psi,\psi]<0$. In this way, we can construct iteratively a sequence $\psi_n\in\mathcal C^\infty_0\left(B_{M_{n+1}}(0)\setminus\overline{B_{M_n}(0)}\right)$ such that $\mathcal Q_v[\psi_n,\psi_n]<0$ and clearly they are linearly independent.\\
For $M>2M_0>0$ fixed we consider a smooth cut-off $\psi=\psi_{M,M_0}$ such that
\begin{equation}\label{psi}
\left\{\begin{array}{ll}\psi\equiv0&\mbox{in }B_{M_0}(0)\cup\left(\mathbb R^2\setminus B_{2M}(0)\right)\\
\psi\equiv1&\mbox{in }A_{2M_0,M}(0)\\
0\le\psi\le1&\mbox{in }\mathbb R^2\\
\int_{\mathbb R^2}|\nabla\psi|^2\le C_0&\mbox{independently on }M_0,M.\end{array}\right.
\end{equation}
Since $v$ is smooth, then
$$\int_{B_{2M_0}(0)}|x|^{2\alpha}e^v<\infty=\int_{\mathbb R^2\setminus B_{2M_0}(0)}|x|^{2\alpha}e^v,$$
therefore for $M$ large enough we will have $2K_0\int_{A_{2M_0,M}(0)}|x|^{2\alpha}e^v>C_0$, hence
\begin{eqnarray*}
\mathcal Q_v[\psi,\psi]&\le&\int_{\mathbb R^2}|\nabla\psi|^2-\int_{A_{2M_0,M}(0)}2K_0|x|^{2\alpha}e^v\psi^2\\
&\le&C_0-2K_0\int_{A_{2M_0,M}(0)}|x|^{2\alpha}e^v<0.
\end{eqnarray*}
\item If $h_0\ge0$, then we choose $\phi_R,\psi_{M,M_0}$ as before with $R$ such that
$$2K_0\int_{B^+_1(0)}|x|^{2\beta}e^v>\int_{\mathbb R^2_+}|\nabla\phi|^2$$
and $M$ such that
$$2K_0\int_{A_{2M_0,M}(0)\cap\mathbb R^2_+}|x|^{2\beta}e^v>\int_{\mathbb R^2_+}|\nabla\phi|^2,$$
so that
\begin{eqnarray*}
\mathcal Q_v[\phi,\phi]&\le&\int_{\mathbb R^2_+}|\nabla\phi|^2-2K_0\int_{\mathbb R^2_+}|x|^{2\beta}e^v\phi^2<0.\\
\mathcal Q_v[\psi,\psi]&\le&\int_{\mathbb R^2_+}|\nabla\psi|^2-2K_0\int_{\mathbb R^2_+}|x|^{2\beta}e^v\psi^2<0.
\end{eqnarray*}
If $h_0<0$, we consider
$$Z_0(s,t):=\log\left(\frac{2t}{1+K_0\left(s^2+\left(t+\frac{h_0}{K_0}\right)^2\right)}-\frac1{h_0}\right),$$
which solves
$$\left\{\begin{array}{ll}
-\Delta Z_0=2 K_0e^{V_0}\left(Z_0+\frac{h_0}{K_0}+\frac1{h_0}\right)&\mbox{in }\mathbb R^2_+\\
\partial_\nu Z_0=h_0e^{V_0/2}Z_0&\mbox{on }\partial\mathbb R^2_+,
\end{array}\right.$$
with
$$V_0(s,t):=\log\frac4{\left(1+K_0\left(s^2+\left(t+\frac{h_0}{K_0}\right)^2\right)\right)^2}.$$
Since, by Theorem \ref{solhalfplane}, $v$ is in the form $v=V_0\circ\left(g-\left(0,\frac{h_0}{K_0}\right)\right)+2\log\frac{|g'|}{|x|^\beta}$, then $z:=Z_0\circ\left(g-\left(0,\frac{h_0}{K_0}\right)\right)$ solves
$$\left\{\begin{array}{ll}
-\Delta z=2 K_0|x|^{2\beta}e^v\left(z+\frac{h_0}{K_0}+\frac1{h_0}\right)&\mbox{in }\mathbb R^2_+\\
\partial_\nu z=h_0|x|^\beta e^{v/2}z&\mbox{on }\partial\mathbb R^2_+
\end{array}\right..$$
We now choose as test function $\phi z$, where $\phi=\phi_R$ is the same cut-off function as in the proof of $(1)$. Since $0<-\frac1{h_0}\le z\le C'$, then
\begin{eqnarray*}
\mathcal Q_v(\phi z)&=&\int_{\mathbb R^2_+}\nabla z\cdot\nabla\left(\phi^2z\right)+\int_{\mathbb R^2_+}\left(z^2|\nabla\phi|^2-2K_0|x|^{2\beta}e^v\phi^2z^2\right)\\
&&-\int_{\partial\mathbb R^2_+}h_0|x|^\beta e^{v/2}\phi^2z^2\\
&=&\int_{\mathbb R^2_+}2K_0|x|^{2\beta}e^v\left(z+\frac{h_0}{K_0}+\frac1{h_0}\right)\phi^2z+\int_{\partial\mathbb R^2_+}h_0|x|^\beta e^{v/2}\phi^2z^2\\
&&+\int_{\mathbb R^2_+}\left(z^2|\nabla\phi|^2-2K_0|x|^{2\beta}e^v\phi^2z^2\right)-\int_{\partial\mathbb R^2_+}h_0|x|^\beta e^{v/2}\phi^2z^2\\
&=&\int_{\mathbb R^2_+}z^2|\nabla\phi|^2+2\left(h_0+\frac{K_0}{h_0}\right)\int_{\mathbb R^2_+}|x|^{2\beta}e^v\phi^2z\\
&\le&\frac{C'^2C_0}{\log R}-2\left(1+\frac{K_0}{h_0^2}\right)\int_{B^+_1(0)}|x|^{2\beta}e^v,
\end{eqnarray*}
which is negative if $M$ is chosen large enough, since the last integral can be arbitrarily large.\\
Similarly, we can show that $\mathcal Q_v(\psi z)<0$ with $\psi=\psi_{M,M_0}$ as in \eqref{psi}, if $M$ is large enough depending on $M_0$.
\end{enumerate}
\end{proof}

\subsection{Global blow-up analysis}

\

The main objective of this sub-section is to provide a proof of Theorem \ref{blowupanal}.\\
The initial step in this process is to appropriately rescale a blowing-up sequence $u_n$ around a local maximum point $x_n$ such that $u_n(x_n)\underset{n\to\infty}\to\infty$. If $x_n\underset{n\to\infty}\to p\in\Sigma\setminus\{p_1,\dots,p_n\}$, then we restrict on a suitably small ball $B_r(x_n)$ and define
\begin{equation}\label{rescal1}
v_n(x):=u_n(\delta_nx+x_n)+2\log\delta_n,\qquad x\in B_\frac r{\delta_n}(0),\qquad\mbox{where }\delta_n:=e^{-\frac{u_n(x_n)}2};
\end{equation}
it is standard to see that $v_n\underset{n\to\infty}\to v$ in $\mathcal C^2_{\mathrm{loc}}$ to an entire solution to \eqref{eqplane}, with $K_0=K(p),\alpha=0$. On the other hand, if $p=p_i$, then we consider
\begin{equation}\label{rescal2}
v_n(x):=u_n(\delta_nx+x_n)+2\log\max\{\delta_n,(1+\alpha)d(x_n,p)\}
\end{equation}
if $\frac{d(x_n,p)}{\delta_n}\underset{n\to\infty}\to\infty$, the rescalement does not see the conical point and the limiting profile is the same as in the previous case; otherwise, we have $v_n\underset{n\to\infty}\to v$ with $v$ solving \eqref{eqplane} with $\alpha=\alpha_p$.\\
If the limiting point is $p\in\partial\Sigma$, then the limiting profiles could either be again a solution on the whole plane to \eqref{eqplane}, if $\frac{d(x_n,\partial\Sigma)}{\delta_n}\underset{n\to\infty}\to\infty$, or a solution on the half-plane to \eqref{eqhalfplane}.\\
We point out that, with respect to classical works, the local maximum points $x_n$ of $u_n$ may be not isolated, therefore $x_n$ could converge to some point different than a given concentration point $p\in\mathcal S$. To overcome this issue, we use Ekeland's variational principle (for details see \cite[Proposition 5.1]{lsmr}).\\

We start by considering the case of finite mass, when the concentration points are only finitely many and we have no residual. The proof is similar to \cite[Lemma 7.4]{lsmr}, therefore we will be sketchy.
\begin{lemma}\label{finitemass}
Under the assumption of Theorem \ref{blowupanal}, if $\int_\Sigma\widetilde K_ne^{u_n}$ is uniformly bounded and the blow-up set $\mathcal S$ defined in \eqref{s}, then $\mathcal S$ is finite and either $\mathcal S=\emptyset$ or
\begin{eqnarray*}
2\widetilde K_ne^{u_n}&\underset{n\to\infty}\rightharpoonup&\sum_{p\in\mathcal S}\gamma(p)\delta_p\\
2\widetilde h_ne^{u_n/2}&\underset{n\to\infty}\rightharpoonup&\sum_{p\in\mathcal S\cap\partial\Sigma}\theta(p)\delta_p
\end{eqnarray*}
for some $\gamma(p),\theta(p)\in\mathbb R$.
\end{lemma}
\begin{proof}
If $\mathcal S\ne\emptyset$, we fix $p\in\mathcal S$ and construct a blow-up sequence as in either \eqref{rescal1} or \eqref{rescal2}. Since $v_n\underset{n\to\infty}\to v$ in $\mathcal C^2_{\mathrm{loc}}$ to a solution $v$ to either \eqref{eqplane} or \eqref{eqhalfplane}, then by Fatou's lemma one has
$$\int_{B_{R_n\delta_n}(x_n)}\widetilde K_ne^{u_n}=\int_{B_{R_n}(0)}\widetilde K_n(\delta_nx+x_n)e^{v_n}\ge\int K_0|x|^{2\alpha}e^v+o(1)\ge\varepsilon_0>0,$$
if $R_n\underset{n\to\infty}\to\infty$ slowly enough; therefore, if $\int_\Sigma\widetilde K_ne^{u_n}$ is uniformly bounded then $\mathcal S$ must be finite.\\
We now suffice to show that $u_n\underset{n\to\infty}\to-\infty$ in $L^\infty_{\mathrm{loc}}\left(\overline\Sigma\setminus\mathcal S\right)$, which will be done as in \cite[Lemma 4]{ls}. We consider, for small $r>0$,
\begin{eqnarray*}
\mbox{either}&\qquad&\sup_{R_n\delta_n\le d(x,x_n)\le r}(u_n(x)+2\log d(x,x_n))\\
\mbox{or}&\qquad&\sup_{R_n\delta_n\le d(x,x_n)\le r}(u_n(x)+2(1+\alpha)\log d(x,x_n)),
\end{eqnarray*}
depending on the blow-up profile we extracted: if the supremum is unbounded, then we can extract a new bubble around a point $x_{n,2}\underset{n\to\infty}\to p$ and repeat, which we can do only a finite number of times $x_{n,1},\dots,x_{n,m}$ because each bubble has at least a fixed amount of mass and the total mass is finite.\\
Now, one can show that for any $\varepsilon>0$ there exists $C=C_\varepsilon$ such that the radial average
$$\overline u_{n,i}(r):=\frac1{|\partial B_r(x_{n,i})|}\int_{\partial B_r(x_{n,i})}u$$
verifies, for any $i=1,\dots,m$,
$$\overline u_{n,i}(r)\le\overline u_{n,i}(R_{n,i}\delta_{n,i})-(4-\varepsilon)\log\frac r{R_{n,i}\delta_{n,i}}+C.$$
Thanks to a Harnack-type inequality, a similar inequality holds for $u_n$, which implies in particular that the mass outside the balls $B_{R_{n,i}\delta_{n,i}}(x_{n,i})$ vanishes and that $u_n\underset{n\to\infty}\to0$ locally uniformly outside $\mathcal S$.
\end{proof}

Quantization of the local blow-up masses follows by a Poho\v zaev identity, as in \cite[Lemma 2.5]{bls1} and \cite[Lemma 7.5]{lsmr}.
\begin{lemma}\label{quantloc}
Under the assumptions of Lemma \ref{finitemass}, we have:
\begin{eqnarray*}
\gamma(p)=8\pi(1+\alpha(p))&&\mbox{if }p\in\mathcal S\cap\Sigma,\\
\gamma(p)+\theta(p)=4\pi(1+\beta(p))&&\mbox{if }p\in\mathcal S\cap\partial\Sigma,
\end{eqnarray*}
where $\alpha(p),\beta(p)$ are defined as in \eqref{alpha},\eqref{beta}.
\end{lemma}
\begin{proof}
We fix $p\in\mathcal S$ and $r<d(p,\mathcal S\setminus\{p\})$ and apply a standard Poho\v zaev identity on $B_r(p)$:
\begin{eqnarray*}
&&\int_{B_r(p)}\left(\frac{8\pi\chi}{|\Sigma|}\nabla u_n\cdot(x-p)+4\left(2\widetilde K_n+\nabla\widetilde K_n\cdot(x-p)\right)e^{u_n}\right)\\
&=&\int_{\partial B_r(p)}\left(4K_ne^{u_n}((x-p)\cdot\nu)+2(\nabla u_n\cdot(x-p))(\nabla u_n\cdot\nu)-|\nabla u_n|^2(x-p\cdot\nu)\right).
\end{eqnarray*}
By the choice of $r$, on $B_r(p)$ we have:
$$2\widetilde K_ne^{u_n}\underset{n\to\infty}\rightharpoonup\gamma(p),\qquad\left(\nabla\widetilde K_n\cdot(x-p)\right)e^{u_n}\underset{n\to\infty}\rightharpoonup\alpha(p)\gamma(p);$$
moreover, since $-\Delta u_n=\gamma(p)\delta_p+O(1)$, then $\nabla u_n=-\frac{\gamma(p)}{2\pi}\frac{x-p}{d(x,p)^2}+O(1)$. Therefore, by letting first $n\to\infty$ and then $r\to0$ we get
$$4(1+\alpha(p))\gamma(p)=\frac{\gamma(p)^2}{2\pi},$$
namely $\gamma(p)=8\pi(1+\alpha(p))$.
We similarly argue on $B^+_r(p)$ if $p\in\partial\Sigma$.
\end{proof}

Finally, we consider the case of infinite mass blow-up, where we exploit the relation with the Morse index which was studied in Subsection \ref{limitprob}.
\begin{lemma}\label{morseinfin}
Under the assumption of Theorem \ref{blowupanal},
$$\mbox{if}\qquad\int_\Sigma\widetilde K_ne^{u_n}\underset{n\to\infty}\to\infty,\qquad\mbox{then}\qquad\mathrm{ind}(u_n)\underset{n\to\infty}\to\infty.$$
\end{lemma}
\begin{proof} We will prove that, if $\mathrm{ind}(u_n)\le m$ for some $m\in\mathbb N$, then $\sup_n\int_\Sigma\widetilde K_ne^{u_n}<\infty$.\\
We first show that $\mathcal S$ has at most $m+1$ elements. Otherwise, if $x_1,\dots,x_{m+2}\in\mathcal S$, then we consider entire solutions $v_1,\dots,v_{m+2}$ to either \eqref{eqplane} or \eqref{eqhalfplane} such that rescalement \eqref{rescal1} or \eqref{rescal2} around $x_i$ converges to $v_i$ for $i=1,\dots,m+2$. Thanks to Theorem \ref{morseindex}, we take $\phi_1,\dots,\phi_{m+2}\in\mathcal C^\infty_0$ such that $\mathcal Q_{v_i}[\phi_i,\phi_i]<0$ and we define $\phi_{n,i}(x):=\phi\left(\frac{x-x_{n,i}}{\delta_{n,i}}\right)$, where $x_{n,i},\delta_{n,i}$ are given by the rescalement around $x_i$. One easily sees that
$$\mathcal I''(u_n)[\phi_{n,i},\phi_{n,i}]=\mathcal Q_{v_i}[\phi_i,\phi_i]+o(1)<0,$$
therefore the index $\mathrm{ind^*}(u_n)$ of $u_n$ associated with $\mathcal I_{\lambda_n}$ (see \eqref{indexi}) is at least $m+2$. From Proposition \ref{morseij} we get $\mathrm{ind}(u_n)\ge m+1$, namely a contradiction.\\
At this point, we suffice to show that the blow-up mass concentrating around each $p\in\mathcal S$ is finite. We claim that any concentration profile $v$ at $p$ has finite mass. Otherwise, from Theorem \ref{morseindex} we can take $\phi_1,\dots,\phi_{m+2}\in\mathcal C^\infty_0$ such that $\mathcal Q_v[\phi_i,\phi_i]<0$ for $i=1,\dots,m+2$; as before, its rescalement around $p$ $\phi_{n,i}(x):=\phi_i\left(\frac{x-x_n}{\delta_n}\right)$ verifies $\mathcal I''(u_n)[\phi_{n,i},\phi_{n,i}]<0$, hence $\mathrm{ind}(u_n)\ge m+1$, a contradiction.\\
Finally, one can only extract finitely many bubbles from each $p\in\mathcal S$: this follows from the fact that, as in the proof of the finiteness of $\mathcal S$, each bubbles raises the Morse index by at least one, so one can have at most $m+1$. As in the proof of Lemma \ref{finitemass}, the total mass will be the sum of the masses accumulating at each $p\in\mathcal S$, which are finitely many and each of which carry a finite amount of mass. The proof is now complete.
\end{proof}

The main result of this section is now an easy consequence of the previous lemmas.
\begin{proof}[Proof of Theorem \ref{blowupanal}]
In case $(1)$ occurs, then \eqref{deltas} follows from Lemmas \ref{finitemass} and \ref{quantloc}; in case $(2)$ occurs, we apply Lemma \ref{morseinfin}.
\end{proof}\

\section{Min-max solutions in the supercritical case}\label{minmaxsol}

In this section, we will prove Theorem \ref{minmax}, that is we will establish an existence result in the supercritical case without any symmetry assumptions. First, we will show that problem \eqref{eqmeanfield} can be viewed as a particular case of a general mean-field problem, consistent with the studies in \cite{d} on closed surfaces.\\
For a fixed $\lambda>0$, we consider the following problem:
\begin{equation}\label{eqlambda}
\left\{\begin{array}{ll}-\Delta u+\frac\lambda{|\Sigma|}=2C^2_\lambda(u)\widetilde Ke^u&\mbox{in }\Sigma\\\partial_\nu u=2C_\lambda(u)\widetilde he^{u/2}&\mbox{on }\partial\Sigma,\end{array}\right.,
\end{equation}
where
$$C_\lambda(u):=\frac\lambda{\sqrt{\left(\int_{\partial\Sigma}\widetilde he^{u/2}\right)^2+2\lambda\int_\Sigma\widetilde Ke^u}+\int_{\partial\Sigma}\widetilde he^{u/2}}.$$
Clearly, \eqref{eqmeanfield} is a particular case of \eqref{eqlambda} with the choice $\lambda=4\pi\chi$. Therefore, Theorem \ref{minmax} will follow once we establish the following result.

\begin{theorem}\label{lambda}
Let $\Sigma$ be a surface with conical singularities and/or corners as specified in \eqref{alphabetacond}, and at least two boundary components, one of which has no corners. If $K>0$ in $\overline\Sigma$ and $\lambda\notin\Gamma$, where $\Gamma$ is as in \eqref{gamma}, then there exists a solution to \eqref{eqlambda}.
\end{theorem}

As in Proposition \ref{meanfieldform}, the solutions to \eqref{eqlambda} are critical points of the energy functional given by
$$\mathcal J_\lambda(u)=\frac12\int_\Sigma|\nabla u|^2-\mathcal F_\frac\lambda{4\pi\chi}\left(\int_\Sigma\widetilde Ke^u,\int_{\partial\Sigma}\widetilde he^{u/2}\right).$$

\subsection{Variational Structure}\label{varstruc}

\

Inspired by \cite{bdm,bjwy}, our argument is based in studying the topology of the energy sublevels
$$\mathcal J_\lambda^a:=\left\{u\in\overline H^1(\Sigma):\mathcal J_\lambda(u)\le a\right\}.$$
The blow-up analysis developed in Section \ref{blowup} is crucial to bypass the failure of the Palais-Smale condition, as illustrated in \cite{lsmr,lsrsr}.

Next, we will show that the low sublevels of $\mathcal J_\lambda$ are not contractible and one can construct a finite-dimensional min-max scheme based on that. To do so, we will prove that functions with low energy must be concentrated at a finite number of points. This will be done using the improved versions of the Trudinger-Moser inequalities introduced in Subsection \ref{min}. We recall that, in view of the assumptions \eqref{alphabetacond}, the Trudinger constant $\tau$ equals $1$.

\begin{lemma}\label{concentr}
Suppose $\lambda<4(k+1)\pi$. Then, for any $\varepsilon,r>0$, there exists $L=L(\varepsilon,r)>0$ such that for any $u\in\mathcal J_\lambda^{-L}$, there exist $k$ points $\{x_1,\dots, x_k\}\subset\overline\Sigma$ such that 
$$\frac{\int_{\bigcup_{i=1}^kB_r(x_i)}\widetilde Ke^u}{\int_\Sigma\widetilde Ke^u}\ge1-\varepsilon.$$
\end{lemma}

\begin{proof}
Assume the inequality does not hold. Then, by a standard covering argument (\cite[Lemma 4.4]{bjmr}), there exist $\delta>0$ and $\Omega_1,\Omega_2,\dots,\Omega_{k+1}\subset\overline\Sigma$ such that 
$$d(\Omega_i,\Omega_j)\ge2\delta\quad\forall i\ne j,\quad\mbox{and}\quad\frac{\int_{\Omega_j}\widetilde Ke^u}{\int_\Sigma\widetilde Ke^u}\ge\delta.$$
For any $j=1,\dots, k+1$, we apply Corollary \ref{tmimprint} and get
$$\log\int_\Sigma\widetilde Ke^u\le\frac{1+\varepsilon}{8(k+1)\pi}\int_\Sigma|\nabla u|^2+C.$$
Moreover, since the ratio $\frac{\widetilde h^2}{\left|\widetilde K\right|}$ is uniformly bounded on $\partial\Sigma$, we have the following inequality (\cite[Proposition 4.1]{bls2})
$$\frac{\left(\int_{\partial\Sigma}\widetilde he^{u/2}\right)^2}{\int_\Sigma\widetilde K e^u}\le C\left(\int_\Sigma|\nabla u|^2+1\right),$$
from which we get
\begin{eqnarray*}
\log\left|\int_{\partial\Sigma}\widetilde he^{u/2}\right|&\le&\frac12\log\int_\Sigma\widetilde Ke^u+\log\left(\int_\Sigma|\nabla u|^2+1\right)+C\\
&\le&\frac{1+2\varepsilon}{16(k+1)\pi}\int_\Sigma|\nabla u|^2+C.
\end{eqnarray*}
Therefore, from \eqref{ineq} and the definition of $\mathcal J_\lambda$ we get
$$\mathcal J_\lambda(u)\ge\frac12\left(1-\frac{(1+2\varepsilon)\lambda}{8(k+1)\pi}\right)\int_\Sigma|\nabla u|^2-C,$$
that is, if $\varepsilon$ is chosen small enough, $\mathcal J_\lambda(u)\ge-L$ for some $L>0$, which is a contradiction.
\end{proof}
Lemma \ref{concentr} implies that the unit measure $\frac{\widetilde Ke^{u_n}}{\int_\Sigma\widetilde Ke^{u_n}}$ resembles, in some sense, a finite linear combination of Dirac deltas with at most $k$ elements. This allows us to describe the low sublevels of $\mathcal J_\lambda$ through the formal barycenters of $\overline\Sigma$ of order $k$:
$$\overline\Sigma_k=\left\{\sum_{i=1}^kt_i\delta_{x_i}:\sum_{i=1}^kt_i=1, x_i\in\overline\Sigma,\forall i=1,\dots, k\right\}.$$
The topological properties of barycenters, which are essential to our purposes, are summed up in the following well-known result, whose proof can be found in \cite{dm,ma}.
\begin{lemma}\label{bary}
For any compact $n$-dimensional manifold $M$, with or without boundary, and $k\ge1$, the set $M_k$ is a stratified set, namely a union of open manifolds of different dimensions, whose maximal dimension is $(n+1)k-1$. Furthermore, $M_k$ is not contractible if $M$ is not contractible.
\end{lemma}
Additionally, the following proposition can be derived from standard arguments based on Lemma \ref{concentr}, (see \cite{d}, Lemma 4.9).
\begin{proposition}\label{proj}
If $\lambda<8(k+1)\pi$, then there exists a projection $\Psi:\mathcal J_\lambda^{-L}\to\overline\Sigma_k$ for $L>0$ large enough.
\end{proposition}

However, barycenters on the entire $\overline\Sigma$ may be too complicated to handle, therefore we will actually consider barycenter on a simpler manifold, such as a curve.\\
We now follow a strategy similar to \cite{bjwy} to construct barycenters on a boundary component of $\Sigma$. We assume that $\Sigma$ has at least two boundary components, one of which has no corners, so that
$$\partial\Sigma=\gamma_1\cup\gamma_2\cup\dots\cup\gamma_l,$$
where $l\ge2$, $\gamma_i\simeq\mathbb S^1$ for $i=1,\dots,l$ and $q_j\notin\gamma_1$ for any $j=1,\dots,m$.

\begin{lemma}\label{retr}
If $\lambda<8(k+1)\pi$, then there exists a projection map $\widetilde\Psi:\mathcal J_\lambda^{-L}\to(\gamma_1)_k$ for $L>0$ large enough.
\end{lemma}

\begin{proof}
From Proposition \ref{proj}, we know that for $L\gg0$, there exists a projection $\Psi:\mathcal J_\lambda^{-L}\to\Sigma_k$. To construct $\widetilde\Psi$, it is sufficient to define a retraction map $\mathcal R:\overline\Sigma\to\gamma_1$, and then set $\widetilde\Psi:=\Psi\circ\mathcal R_*$, where $\widetilde\Psi(u):=\sum_{i=1}^kt_i\delta_{\mathcal R(x_i)}$ if $\Psi(u)=\sum_{i=1}^kt_i\delta_{x_i}$.\\
Consider an open subset $\Omega\subset\Sigma$ obtained by joining tubular neighborhoods of $\gamma_1$ and $\gamma_2$ together with tubular neighborhoods of a geodesic curve connecting these boundary components in $\Sigma$, as shown in Figure \ref{fig}.\\
\begin{figure}[ht]
\centering
\includegraphics[width=9cm]{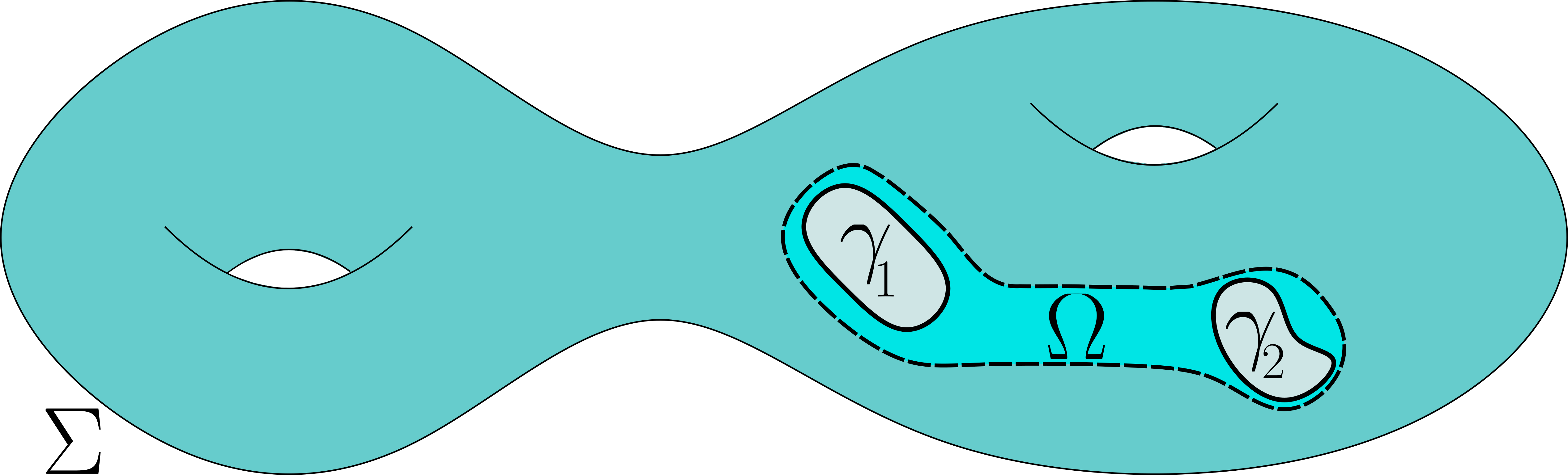}
\caption{The open domain $\Omega\subset\Sigma $.}
\label{fig}
\end{figure}
Next, we define the surface $\widehat\Sigma$ by identifying $\overline\Sigma\setminus(\Omega\cup\gamma_1\cup\gamma_2)$ to a single point. This identification is achieved via the quotient map $p:\overline\Sigma\to\widehat\Sigma$, which topologically collapses the region $\overline\Sigma\setminus(\Omega\cup\gamma_1\cup\gamma_2)$ to a single point. The resulting surface is orientable of genus zero with two boundary components, hence it is topologically equivalent to a standard cylinder, i.e. $\widehat\Sigma\approx\mathbb S^1\times[0,1]$, with boundary $\gamma_1\cup\gamma_2\approx\left(\mathbb S^1\times\{0\}\right)\cup\left(\mathbb S^1\times\{1\}\right)$.\\ 
We define the retraction by quotienting $\Sigma$ to the cylinder $\widehat\Sigma$ and then mapping the latter on one boundary. More precisely, lef $f:\widehat\Sigma\leftrightarrow\mathbb S^1\times[0,1]$ be the above-mentioned homeomorphism and $\pi:\mathbb S^1\times[0,1]\to\mathbb S^1\times\{0\}$ the projection onto the first component,$\mathcal R:\overline\Sigma\to\gamma_1$ is then defined as
$$\mathcal R:=f^{-1}\circ\pi\circ f\circ p.$$
\end{proof}

To further explore the topological properties of the low sublevels, we will construct a reverse map $\Phi:(\gamma_1)_k\to\mathcal J_\lambda^{-L}$. To achieve this, we consider a family of standard \emph{bubbles} centered on the boundary component $\gamma_1$ and we define, for $\Lambda>0$, $\Phi:(\gamma_1)_k\to\overline H^1(\Sigma)$ as 

\begin{equation}\label{phi}
\Phi(\sigma):=\varphi_{\Lambda,\sigma}-\overline{\varphi_{\Lambda,\sigma}},
\end{equation}
$$\mbox{where}\quad\varphi_{\Lambda,\sigma}(x)=\log\sum_{i=1}^kt_i\frac{\Lambda^2}{\left(1+\Lambda^2d(x,x_i)^2\right)^2}\qquad\mbox{for}\quad\sigma=\sum_{i=1}^kt_i\delta_{x_i}.$$

In this context, the choice of $\gamma_1$ is crucial, because we are working with bubbles centered on a boundary component without any conical points nor corners, hence we basically can neglect singularities. Therefore, we can argue as in \cite[Proposition 4.2]{ma} to derive the following estimates.

\begin{lemma}\label{bubble}
Let $\varphi_{\Lambda,\sigma}$ be defined as in \eqref{phi}. Then, as $\Lambda\to\infty$, we have the following estimates, uniformly on $\sigma\in(\gamma_1)_k$:
\begin{eqnarray*}
\frac12\int_\Sigma|\nabla\varphi_{\Lambda,\sigma}|&\le&8k\pi(1+o(1))\log\Lambda,\\
\log\int_\Sigma\widetilde Ke^{\varphi_{\Lambda,\sigma}-\overline{\varphi}_{\Lambda,\sigma}}&=&2(1+o(1))\log\Lambda,\\
\log\left|\int_{\partial\Sigma}\widetilde he^{(\varphi_{\Lambda,\sigma}-\overline{\varphi}_{\Lambda,\sigma})/2}\right|&=&(1+o(1))\log\Lambda.
\end{eqnarray*}
Moreover, we have 
$$\frac{\widetilde Ke^{\varphi_{\Lambda,\sigma}}}{\int_\Sigma\widetilde Ke^{\varphi_{\Lambda,\sigma}}}\underset{\Lambda\to\infty}\rightharpoonup\sigma,$$
in the sense of measures.
\end{lemma}
The above estimates indicate that the values of the functional can be made as negatively large as desired, leading to the following conclusion.

\begin{proposition}\label{test}
If $\lambda>8k\pi$, then, for any $L>0$, there exists $\Lambda=\Lambda(L)\gg0$ such that $\Phi((\gamma_1)_k)\subset\mathcal J_\lambda^{-L}$, where $\Phi$ is defined as in \eqref{phi}.
\end{proposition}

\begin{proof}
By the previous lemma we get
$$\mathcal J_\lambda\left(\Phi(\sigma)\right)\le(8k\pi-\lambda+o(1))\log\Lambda$$
which, since $\lambda>8k\pi$, will have an arbitrarily large negative value by taking $\Lambda$ large enough.
\end{proof}

\subsection{Proofs of the Theorems \ref{lambda} and \ref{minmax}}\label{proof}

\

In order to conclude, we will exploit the compactness results from Theorem \ref{blowupanal} and Corollary \ref{compa}.\\
In particular, with respect to classical mean field problem, in order to prevent blow-up we will also need to take account of the boundedness of the Morse index. We get a crucial upper bound on the Morse index of solutions due to the fact that solutions come from a finite-dimensional min-max scheme, in the spirit of the work of Fang and Ghoussoub \cite{fg}.\\
We first construct a min-max scheme based on the analysis of the barycenters $\left(\gamma_1\right)_k$. Then, we introduce an auxiliary parameter $\mu$ and apply the celebrated monotonicity trick by Struwe \cite{s}; we get a solution to \eqref{eqlambda} for a dense sequence of values $\mu_n$ with uniformly bounded Morse index. Finally, we pass to the limit for $\mu_n\underset{n\to\infty}\to1$ using the compactness results we already got. The introduction of the extra parameter is a new ingredient for this type of problems; it is essential, since we need the energy functional to be monotone with respect to the parameter. Similar arguments have been used in other curvature prescription problems \cite{dm,lsmr,lsrsr}.\\

In order to precisely define the min-max structure, we define the contractible cone based on the barycenter set $(\gamma_1)_k$ as
$$\widehat{(\gamma_1)_k}:=\frac{(\gamma_1)_k\times [0,1]}{(\gamma_1)_k\times\{0\}}.$$
We then consider the family of maps on $\widehat{(\gamma_1)_k}$ which extend to the cone the bubbles defined on $(\gamma_1)_k$ as \eqref{phi}:
$$\Pi_\Lambda=\left\{\pi:\widehat{(\gamma_1)_k}\to H^1(\Sigma):\mbox{ $\pi$ continuous and }\pi\left(\sigma\times\{1\}\right)=\varphi_{\Lambda,\sigma}\,\forall\sigma\in(\gamma_1)_k\right\}.$$
Following the same strategy as \cite{ma}, we then obtain the next min-max structure.

\begin{proposition}\label{infsup}
Assume $8k\pi<\lambda<8(k+1)\pi$ and $L\gg0$ so large that Lemma \ref{retr} holds true with $\frac L4$. Then, the set $\Pi_\Lambda$ is non-empty and
$$\overline\Pi_\Lambda:=\inf_{\pi\in\Pi_\Lambda}\sup_{(\sigma,t)\in\widehat{(\gamma_1)_k}}\mathcal J_\lambda\left(\pi(z,t)\right)>-\frac L2.$$
\end{proposition}

\begin{proof}
To show that $\Pi_\Lambda\ne\emptyset$, hence $\overline\Pi_\Lambda<\infty$, we observe that the map 
$$\overline\pi(\sigma,t)=t\varphi_{\Lambda,\sigma},\quad\mbox{for }(\sigma,t)\in\widehat{(\gamma_1)_k},$$
is indeed an element of $\Pi_\Lambda$.\\
We assume, for the sake of contradiction, that $\overline\Pi_\Lambda\le-\frac L2$. By the assumptions, Lemma \ref{retr} and Proposition \ref{test}, there exist maps
$$\widetilde\Psi:\mathcal J_\lambda^{-L/4}\to(\gamma_1)_k,\qquad\Phi:(\gamma_1)_k\to\mathcal J_\lambda^{-L/4},$$
whose composition $\widetilde\Psi\circ\Phi$ is homotopically equivalent to the identity on $(\gamma_1)_k$; the homotopical equivalence is obtained by letting $\Lambda\to\infty$, thanks to the last statement in Lemma \ref{bubble}.\\
If $\overline\Pi_\Lambda\le\frac L2$, then there exists $\pi\in\Pi_\Lambda$ such that $\mathcal J_\lambda(\pi(\sigma,t))\le-\frac38L$ for any $(\sigma,t)\in\widehat{(\gamma_1)_k}$. Therefore, one may define the map $H(\sigma,t):\widehat{(\gamma_1)_k}\to(\gamma_1)_k$ as $H(\sigma,t):=\widetilde\Psi(\pi(\sigma,t))$: this would be a homotopy between $\widetilde\Psi\circ\varphi_{\Lambda,\sigma}$ and a constant map, but this is impossible, since the former map is homotopically equivalent to the identity and $(\gamma_1)_k$ is not contractible by Lemma \ref{bary}. We found a contradiction.
\end{proof}

In order to apply the Struwe's monotonicity trick, we need to introduce the following perturbed problem:
\begin{equation}\label{mulambda}
\left\{\begin{array}{ll}-\mu\Delta u+\frac\lambda{|\Sigma|}=2C^2_\lambda(u)\widetilde Ke^u&\mbox{in }\Sigma,\\
\partial_\nu u=2C_\lambda(u)\widetilde he^{u/2}&\mbox{on }\partial\Sigma.\end{array}\right.,
\end{equation}
whose energy functional is given by $$\mathcal J_{\lambda,\mu}(u)=\frac\mu2\int_\Sigma|\nabla u|^2-\mathcal F_{\frac\lambda{4\pi}}\left(\int_{\partial\Sigma}\widetilde he^{u/2},\int_\Sigma\widetilde Ke^u\right).$$
It is easy to see that, if $\mu$ is close enough to $1$, then all estimates from Subsection \ref{varstruc}, as well as Proposition \ref{infsup}, holds true also for $\mathcal J_{\lambda,\mu}$ uniformly in $\mu$. In particular, given any large number $L>0$, there exists a $\Lambda$ so large that 
\begin{eqnarray*}
\sup_{(\sigma,t)\in\partial\widehat{(\gamma_1)_k}}\mathcal J_{\lambda,\mu}\left(\pi(\sigma,t)\right)&<&-2L,\\
\overline\Pi_{\Lambda,\mu}:=\inf_{\pi\in\Pi_\Lambda}\sup_{(\sigma,t)\in\widehat{(\gamma_1)_k}}\mathcal J_{\lambda,\mu}\left(\pi(\sigma,t)\right)&>&-\frac L2,\\
\overline\Pi_{\Lambda,\mu}&\le&\overline L,
\end{eqnarray*}
for some $\overline L>0$.\\

Now, applying the Monotonicity trick of Struwe, we can argue as in \cite{ma} to obtain the following assertions. 
\begin{proposition}
Under the assumptions of Theorem \ref{lambda} there exist $\varepsilon_0>0$ such that, for a dense set of values $\mu\in[1-\varepsilon_0,1+\varepsilon_0]$, \eqref{mulambda} has a solution $u_\mu$ whose energy and Morse index satisfy
$$\mathcal J_{\lambda,\mu}(u_\mu)=\overline\Pi_{\Lambda,\mu},\qquad\mathrm{ind}(u_\mu)\le 2k.$$ 
\end{proposition}
\begin{proof}
The existence of a solution for a dense set of parameters at the min-max energy level follow by arguing as in \cite{s}. The upper bound on the Morse index is a consequence of Theorem 2 of \cite{bcjs}, since the min-max scheme is based on $\widehat{(\gamma_1)_k}$, which has dimension $2k$ by Lemma \ref{bary}.
\end{proof}
We can now prove Theorem \ref{lambda} and, as its particular case, Theorem \ref{minmax}.
\begin{proof}[Proof of Theorem \ref{lambda}] We take a sequence $\mu_n\underset{n\to\infty}\to1$ such that the perturbed problem \eqref{mulambda} has a solution $u_n$. Since $\lambda\notin\Gamma$ and $\sup_n\mathrm{ind}(u_n)<\infty$, by Corollary \ref{compa} we get that $u_n$ is uniformly bounded in $L^\infty(\Sigma)$. By standard estimates, $u_n$ converges in $\mathcal C^{2,\gamma}\left(\overline\Sigma\right)$ to some $u$ which solves \eqref{eqlambda}.\\
In the case $\lambda=4\pi\chi$ we get a solution to \eqref{eqsing}.
\end{proof}\

\section{The case of non-positive singular Euler characteristic}

This section is devoted to the proof of Theorems \ref{chi=0} and \ref{chi<0}, that is the case $\chi\le0$.\\
Such results have counterparts in the regular case from the paper \cite{bls2} and the proofs are rather similar. In contrast with the case $\chi>0$, here in the definition of the energy functional the logarithmic term does not play a significant role (see \eqref{j} and \eqref{fab}), but it is crucial to control the rational term. A fundamental role is played by the function $\mathfrak D$, defined as in by \eqref{d}, which interestingly is unaffected by the presence of conical singularities or corners.\\
In order to control the nonlinear terms by means of the quadratic term, we need the following result, see \cite[Proposition 4.1]{bls2}.
\begin{proposition}\label{trace}
Assume $\widetilde K(x)\lneqq0$ for any $x\in\Sigma$ and let $\mathfrak D$ be defined by \eqref{d}. Then for any $\varepsilon>0$ there exists $C=C_\varepsilon$ such that for any $u\in H^1(\Sigma)$ one has:
$$\frac{\left(\int_{\partial\Sigma}\widetilde he^{u/2}\right)^2}{\int_\Sigma|\widetilde K|e^u}\le\frac{(\mathfrak D_M+\varepsilon)^2}4\int_\Sigma|\nabla u|^2+C\,\quad\mbox{where }\mathfrak D_M:=\max_{\partial\Sigma}|\mathfrak D|.$$
\end{proposition}

\begin{proof}[Proof of Theorem \ref{chi=0}]
We recall that we study the energy functional \eqref{j} in the space
$$H_0=\left\{u\in\overline H^1(\Sigma):
\int_\Sigma\widetilde Ke^u\int_{\partial\Sigma}\widetilde he^{u/2}<0\right\}.$$
We study separately each of the three cases:
\begin{enumerate}
\item Since $\widetilde K>0$ somewhere and $\int_{\partial\Sigma}\widetilde he^{u/2}<0$ for all $u$, then
$$H_0=\left\{u\in\overline H^1(\Sigma):
\int_\Sigma\widetilde Ke^u>0\right\}\ne\emptyset.$$
Therefore,
$$\mathcal J(u)=\frac12\int_\Sigma|\nabla u|^2+2\frac{\left(\int_{\partial\Sigma}\widetilde he^{u/2}\right)^2}{\int_\Sigma\widetilde Ke^u}\ge\frac12\int_\Sigma|\nabla u|^2\underset{\|u\|\to\infty}\to\infty.$$
To conclude with the coercivity of $\mathcal J$ in $H_0$ we need to verify that the energy also diverges on the boundary
$$\partial H_0=\left\{u\in H_0:
\int_\Sigma\widetilde Ke^u=0\right\}.$$
If $u_n\underset{n\to\infty}\rightharpoonup u_0\in\partial H_0$, then by weak continuity, one has
$$\int_\Sigma\widetilde Ke^{u_n}\underset{n\to\infty}\to\int_\Sigma\widetilde Ke^{u_0}=0\qquad\mbox{and}\qquad\int_{\partial\Sigma}\widetilde h_ne^{u_n/2}\underset{n\to\infty}\to\int_{\partial\Sigma}\widetilde he^{u/2}<0,$$
therefore $$\mathcal J(u_n)\ge\frac{\left(\int_{\partial\Sigma}\widetilde he^{u_n/2}\right)^2}{\int_\Sigma\widetilde Ke^{u_n}}\underset{n\to\infty}\to\infty,$$
which shows coercivity, hence existence of minimizing solutions.
\item Since $\widetilde K\ge0$, by the same argument as before we show coercivity on the whole $\overline H^1(\Sigma)$ and the existence of minimizers solving
\begin{equation}\label{noc}
\left\{\begin{array}{ll}-\Delta u=2C^2(u)\widetilde Ke^u&\mbox{in }\Sigma\\\partial_\nu u=2C(u)\widetilde he^{u/2}&\mbox{on }\partial\Sigma\end{array}\right.,
\end{equation}
without any information on the sign of $C(u)$ (since we are on $\overline H^1(\Sigma)$). To get a solution to \eqref{eqmeanfield}, hence to \eqref{eqsing}, we need to show that $C(u)>0$.\\
If $C(u)=0$, then the only solution to \eqref{noc} must be $u\equiv0$, hence we get the following contradiction
$$0=C(u)=-\frac{\int_{\partial\Sigma}\widetilde he^{u/2}}{\int_\Sigma\widetilde Ke^u}=-\frac{\int_{\partial\Sigma}\widetilde h}{\int_\Sigma\widetilde K}>0.$$
Finally, if $C(u)<0$ we get a solution to \eqref{eqmeanfield} with $-\widetilde h$ instead of $\widetilde h$, since \eqref{noc} is invariant by change of sign to $\widetilde h$; however, problem
$$\left\{\begin{array}{ll}-\Delta u=2\widetilde Ke^u&\mbox{in }\Sigma,\\
\partial_\nu u=-2\widetilde he^{u/2}&\mbox{on }\partial\Sigma
\end{array}\right.$$
has no solutions because multiplying by $e^{-u/2}$ and integrating by parts one gets again a contradiction:
$$-\int_{\partial\Sigma}\widetilde h<\int_\Sigma\widetilde Ke^u-\int_{\partial\Sigma}\widetilde he^{u/2}=-\frac12\int|\nabla u|^2e^{-u/2}<0.$$
Therefore, it must be $C(u)>0$, hence one has a true solution to \eqref{eqmeanfield}.
\item Since $\int_\Sigma\widetilde Ke^u<0<\int_{\partial\Sigma}\widetilde he^{u/2}$, therefore $H_0=\overline H^1(\Sigma)$, hence, in order to get minimizing solutions, we suffice to show coercivity at infinity. To this purpose, we take small $\varepsilon$ such that $\mathfrak D(y)+\varepsilon<1$ for any $y\in\partial\Sigma$ and apply Proposition \ref{trace}:
$$\mathcal J(u)=\frac12\int_\Sigma|\nabla u|^2-2\frac{\left(\int_{\partial\Sigma}\widetilde he^{u/2}\right)^2}{\int_\Sigma|\widetilde K|e^u}\ge\frac{1-(\mathfrak D_M+\varepsilon)^2}2\int_\Sigma|\nabla u|^2-C_\varepsilon\underset{\|u\|\to\infty}\to\infty.$$
\end{enumerate}
\end{proof}

In the case $\chi<0$ we have an extra logarithmic term in the energy, which was not present when $\chi=0$, but it does not really affect the problem. This follows from the following elementary result.
\begin{lemma}\label{f}
For any $\varepsilon>0$ there exists $C_\varepsilon>0$ such that the function
\begin{equation}\label{ft}
f(t):=F_\chi(-1,t)=8\pi|\chi|\left(-\log{\left(\sqrt{t^2+8\pi|\chi|}-t\right)}+\frac{t}{\sqrt{t^2+8\pi|\chi|}-t}\right)
\end{equation}
verifies
$$f(t)\le(2+\varepsilon)t_+^2+C_\varepsilon.$$
\end{lemma}
In view of Lemma \ref{f}, the case $\chi<0$ can still be treated using Proposition \ref{trace}, similarly as Theorem \ref{chi=0}.
\begin{proof}[Proof of Theorem \ref{chi<0}]
We write the energy functional as 
$$\mathcal J(u)=\frac12\int_\Sigma|\nabla u|^2+4\pi|\chi|\log{\left(-\int_\Sigma\widetilde Ke^u\right)}+4\pi|\chi|\log{\left(8\pi\right)}-f\left(\frac{\int_{\partial\Sigma}\widetilde he^{u/2}}{\sqrt{-\int_\Sigma\widetilde Ke^u}}\right),$$
where $f$ is as in \eqref{ft}.\\
The first nonlinear term is uniformly bounded from below on $H_\chi$ because of the Jensen's inequality
$$\int_\Sigma-\widetilde Ke^u=\int_\Sigma e^{u+\log{|\widetilde K|}}\ge|\Sigma|\int_\Sigma e^{\frac1{|\Sigma|}\int_\Sigma(u+\log{|\widetilde K|})}=|\Sigma|e^{\frac1{|\Sigma|}\int_\Sigma\log{|\widetilde K|}}>-C.$$
Since $\mathfrak D^+_M:=\max_{\partial\Sigma}{\frac h{\sqrt{|K|}}}<1,$ using Lemma \ref{f} and Proposition \ref{trace} with $h_+$ instead of $h$ we get, for $\varepsilon$ small enough:
\begin{eqnarray*}
\mathcal J(u)&\ge&\frac12\int_\Sigma|\nabla u|^2+f\left(\frac{\int_{\partial\Sigma}\widetilde he^{u/2}}{\sqrt{-\int_\Sigma\widetilde Ke^u}}\right)-C\\
&\ge&\frac12\int_\Sigma|\nabla u|^2-(2+\varepsilon)\frac{\left(\int_{\partial\Sigma}\widetilde he^{u/2}\right)_+^2}{-\int_\Sigma\widetilde Ke^u}-C\\
&\ge&\frac{2-(2+\varepsilon)(\mathfrak D_M^++\varepsilon)^2}4\int_\Sigma|\nabla u|^2-C\underset{\|u\|\to\infty}\to\infty,
\end{eqnarray*}
hence $\mathcal J$ has minimizers which solve \eqref{eqmeanfield}.
\end{proof}\

\section*{Acknowledgments}

L. B. would like to express his gratitude to the University of Granada
for the kind hospitality received during his visit in September 2024.\\
F. J. R.-S. would like to express his gratitude to the Mathematics and Physics Department of Roma
Tre University for the kind hospitality received during his stay from April to June 2024.\\
Both authors wish to thank Professors David Ruiz and Rafael Lopez-Soriano for various discussions and suggestions that have been of great help in the preparation of this work.\\
Authors are also thankful to the anonymous referee for the careful reading and for many useful suggestions.\

\bibliography{references.bib} 
\bibliographystyle{plain}
\end{document}